\documentclass{amsart}  \textwidth=14.5cm \oddsidemargin=1cm
\usepackage{
    amsmath,
    amsfonts,
    amssymb,
    amsthm,
    amscd,
    comment,
    enumitem,
    etoolbox,
    textcomp,   
    gensymb,    
    mathtools,
    mathdots,
    stmaryrd    
}
\usepackage[dvipsnames]{xcolor}
\usepackage{tipa}

\usepackage[T1]{fontenc}
\usepackage{dsfont}             
\usepackage[colorlinks=true, linkcolor=blue, citecolor=blue, urlcolor=blue, breaklinks=true]{hyperref}

\usepackage{wasysym}
\usepackage{xypic,makecell,hhline}

\DeclareFontFamily{OT1}{pzc}{}
\DeclareFontShape{OT1}{pzc}{m}{it}{<-> s * [1.10] pzcmi7t}{}
\DeclareMathAlphabet{\mathpzc}{OT1}{pzc}{m}{it}

\usepackage{zref-clever}

\zcsetup{
  cap,                      
  nameinlink=false,         
  lastsep = {, and }        
}

\zcRefTypeSetup{assumption}{
    Name-sg = Assumption ,
    name-sg = assumption ,
    Name-pl = Assumptions ,
    name-pl = assumptions ,
}
\zcRefTypeSetup{cor}{
    Name-sg = Corollary ,
    name-sg = corollary ,
    Name-pl = Corollaries ,
    name-pl = corollaries ,
}
\zcRefTypeSetup{defin}{
    Name-sg = Definition ,
    name-sg = definition ,
    Name-pl = Definitions ,
    name-pl = definitions ,
}
\zcRefTypeSetup{enumi}{
    Name-sg = {} ,
    name-sg = {} ,
    Name-pl = {} ,
    name-pl = {} ,
}
\zcRefTypeSetup{equation}{
    Name-sg = {} ,
    name-sg = {} ,
    Name-pl = {} ,
    name-pl = {} ,
}
\zcRefTypeSetup{eg}{
    Name-sg = Example ,
    name-sg = example ,
    Name-pl = Examples ,
    name-pl = examples ,
}
\zcRefTypeSetup{lem}{
    Name-sg = Lemma ,
    name-sg = lemma ,
    Name-pl = Lemmas ,
    name-pl = lemmas ,
}
\zcRefTypeSetup{prop}{
    Name-sg = Proposition ,
    name-sg = proposition ,
    Name-pl = Propositions ,
    name-pl = propositions ,
}
\zcRefTypeSetup{rem}{
    Name-sg = Remark ,
    name-sg = remark ,
    Name-pl = Remarks ,
    name-pl = remarks ,
}
\zcRefTypeSetup{theo}{
    Name-sg = Theorem ,
    name-sg = theorem ,
    Name-pl = Theorems ,
    name-pl = theorems ,
}

\zcsetup{
  countertype = {
    enumi=item, enumii=item, enumiii=item, enumiv=item
  },
  counterresetby = {
    enumii=enumi, enumiii=enumii, enumiv=enumiii
  }
}
\zcRefTypeSetup{item}{
    Name-sg = {} ,
    name-sg = {} ,
    Name-pl = {} ,
    name-pl = {} ,
}

\AddToHook{env/assumption/begin}{\zcsetup{countertype={theo=assumption}}}
\AddToHook{env/cor/begin}{\zcsetup{countertype={theo=cor}}}
\AddToHook{env/defin/begin}{\zcsetup{countertype={theo=defin}}}
\AddToHook{env/eg/begin}{\zcsetup{countertype={theo=eg}}}
\AddToHook{env/lem/begin}{\zcsetup{countertype={theo=lem}}}
\AddToHook{env/prop/begin}{\zcsetup{countertype={theo=prop}}}
\AddToHook{env/rem/begin}{\zcsetup{countertype={theo=rem}}}

\newcommand{\cref}[1]{\zcref{#1}}
\newcommand{\Cref}[1]{\zcref[S]{#1}}

\makeatletter
\def\namedlabel#1#2{\begingroup
    (#2)%
    \def\@currentlabel{(#2)}%
    \phantomsection\label{#1}\endgroup
}
\makeatother

\newcommand\C{\mathbb{C}}

\newcommand\Z{\mathbb{Z}}
\newcommand\kk{\Bbbk}

\newcommand{\ck}{\mathfrak{k}}

\newcommand{\ns}{\text{ns}}

\newcommand{\Ub}{\Ualg_{\bu}}
\newcommand{\wt}{\text{wt}}

\newcommand{\Fi}{\mathcal F_\io}

\newcommand{\Rem}{\text{Rem}}
\newcommand{\tUi}{\widetilde{\Ualg}^\io}
\newcommand{\tB}{\tilde{B}}

\newcommand{\fg}{\mathfrak{g}}
\newcommand{\fh}{\mathfrak{h}}
\newcommand{\fgl}{\mathfrak{gl}}

\newcommand{\osp}{\mathfrak{osp}}

\newcommand{\Ieven}{I_{\overline{0}}}
\newcommand{\Iodd}{I_{\overline{1}}}

\newcommand{\io}{\imath}
\newcommand{\va}{\varsigma}

\newcommand{\Ualg}{\bU}
\newcommand{\Ui}{\Ualg^\imath}

\newcommand{\Uio}{\Ualg^{\imath0}}

\DeclareMathOperator{\End}{End}

\DeclareMathOperator{\id}{id}

\usepackage{tikz}
\usetikzlibrary{arrows.meta,patterns}
\usetikzlibrary{decorations.markings,decorations.pathreplacing}
\usepackage{tikz-cd}

\tikzset{multi/.style={very thick}}
\tikzset{
    anchorbase/.style={
        >=To,
        line cap = round, line join = round,
        baseline={([yshift=-0.5ex]current bounding box.center)},
    }
}
\tikzset{
    centerzero/.style={
        >=To,
        line cap = round, line join = round,
        baseline={([yshift=-0.5ex](#1))},
        },
    centerzero/.default={0,0}
}
\tikzset{   
    wipe/.style={
        white,
        line cap = butt,    
        line width=4pt
    }
} 
\tikzset{
    over/.style={
        preaction={
            draw=white,
            line cap = butt,    
            line width=4pt,
            -{},                
        }
    },
}
\tikzset{->-/.style={decoration={
    markings,
    mark=at position #1 with {\arrow{>}}},postaction={decorate}}
}
\tikzset{-<-/.style={decoration={
    markings,
    mark=at position #1 with {\arrow{<}}},postaction={decorate}}
}

\tikzset{
    circ/.style={
        draw,
        circle,
        minimum size=0.7em,
        inner sep=0pt,
    }
}
\tikzset{
    circslash/.style={
        draw,
        circle,
        minimum size=0.7em,
        inner sep=0pt,
        append after command={
            \pgfextra{
                \draw (\tikzlastnode.south west) -- (\tikzlastnode.north east);
            }
        }
    }
}
\tikzset{
    circcross/.style={
        draw,
        circle,
        minimum size=0.7em,
        inner sep=0pt,
        append after command={
            \pgfextra{
                \draw (\tikzlastnode.south west) -- (\tikzlastnode.north east);
                \draw (\tikzlastnode.south east) -- (\tikzlastnode.north west);
            }
        }
    }
}
\tikzset{
    circblack/.style={
        draw=black,
        fill=black,
        circle,
        minimum size=0.7em,
        inner sep=0pt,
    }
}

\newtheorem{theo}{Theorem}[section]

\newtheorem{prop}[theo]{Proposition}
\newtheorem{lem}[theo]{Lemma}

\theoremstyle{definition}

\newtheorem{defin}[theo]{Definition}
\newtheorem{rem}[theo]{Remark}
\newtheorem{eg}[theo]{Example}

\numberwithin{equation}{section}
\allowdisplaybreaks

\setenumerate[1]{label=(\alph*)}          

\newtoggle{comments}
\newtoggle{details}
\newtoggle{detailsnote}

\toggletrue{comments}   
\toggletrue{details}   

\iftoggle{comments}{%

    \newcommand{\question}[1]{
        \ \\
        {\color{blue}
            \textbf{Question:} #1
        }
        \ \\
    }
    }{%
        \newcommand{\acomments}[1]{}
        \newcommand{\ycomments}[1]{}
        \newcommand{\question}[1]{}
    }

\iftoggle{details}{%
    \newcommand{\details}[1]{
        \ \\
        {\color{OliveGreen}
            \textbf{Details:} #1
        }
        \\
    }
}{%
    \newcommand{\details}[1]{\ignorespaces}
}

\newcommand{\I}{I}

\newcommand{\bu}{\bullet}
\newcommand{\bU}{\mathbf{U}}

\newcommand{\gl}{\mathfrak{gl}}

\newcommand{\Vset}{I_V}
\newcommand{\qbinom}[2]{\begin{bmatrix} #1\\#2 \end{bmatrix} }

\begin{document}

\title{A Serre-type presentation for $\imath$quantum supergroups of type A}

\author{Weideng Cui}
\address[Weideng Cui]{School of Mathematics, Shandong University, Jinan, Shandong 250100, China} 
\email{cwdeng@amss.ac.cn}

\author{Yaolong Shen}
\address[Yaolong Shen]{School of Mathematical Sciences, Key Laboratory of MEA (Ministry of Education)
\& Shanghai Key Laboratory of PMMP, East China Normal University, Shanghai 200241,
China}
\email{yaolongshen3555@gmail.com}

\maketitle

\begin{abstract}
We establish a Serre-type presentation for $\imath$quantum enveloping superalgebras arising from quantum supersymmetric pairs of type A. Using the projection technique, we first prove that for quantum supersymmetric pairs of basic types, evaluating a quantum Serre polynomial on the coideal generators yields terms of strictly lower weight in the natural filtration. We then compute these remainder terms for all type A super Satake diagrams.  In particular, we determine the new degree-four relations associated with isotropic odd simple roots of the relevant local Satake diagrams.
\end{abstract}

\tableofcontents

\section{Introduction}

Quantum symmetric pairs are quantum analogues of symmetric pairs of Lie
algebras. Algebraically, they consist of a quantum group $\Ualg$ together
with a distinguished right coideal subalgebra $\Ui$, usually called an
$\imath$quantum group.  The finite-type theory was developed systematically
by Letzter \cite{Let99,Let02} and extended to the Kac--Moody setting by Kolb
\cite{Kol14}. The theory of quantum symmetric pairs has since become an important meeting point of quantum groups and representation theory: $\imath$quantum groups
support analogues of canonical bases, Schur dualities, intertwining
operators, Hall algebra realizations, categorifications, and Drinfeld-type current
presentations; see \cite{Wan23} for a survey. Explicit generators and
relations are basic to all of these directions. They provide a starting point for studying its representations and symmetries.

The super setting introduces phenomena that are absent from the ordinary
theory. A Lie superalgebra may admit many inequivalent choices of simple
roots, and the parity of those roots affects both the signs and the form of
the defining relations. For $\fg=\fgl(m|n)$, the quantum enveloping
superalgebra associated with an arbitrary Dynkin diagram was constructed by
Yamane \cite{Yam94,Yam99}. Besides the usual quadratic and cubic quantum
Serre relations, its presentation contains higher-order relations centered
at isotropic odd simple roots. 

Quantum supersymmetric pairs of type~sAIII were introduced in \cite{CL23,She25},
where an $\imath$Schur duality with the Hecke algebra of type~B and the
corresponding quasi $K$-matrix were established. The construction was
subsequently extended in \cite{SW25} to all basic Lie superalgebras; see also \cite{AMS25}. Further developments based on quantum supersymmetric pairs can be found in \cite{L25,L27,SSS26,SZ26}. In these frameworks, a super Satake diagram
$(\I=\I_\circ\sqcup\I_\bu,\tau)$ determines a right coideal subalgebra
$\Ui\subseteq\Ualg$ generated by the Levi part, a toral subalgebra, and
elements
\[
 B_i=F_i+\va_iT_{w_\bu}(E_{\tau i})K_i^{-1}+\kappa_iK_i^{-1},
 \qquad i\in\I_\circ.
\]
The quantum Iwasawa decomposition and a PBW-type basis theorem proved in
\cite{SW25} imply the existence of a presentation, but they do not by
themselves identify the lower-order corrections to the Serre relations.
Determining those corrections explicitly is the problem addressed in this
paper.

The underlying difficulty can be stated succinctly. If $p$ is a quantum
Serre polynomial for $\Ualg$, then $p(\underline F)=0$, whereas
$p(\underline B)$ need not vanish. Its leading filtered component does
vanish, so $p(\underline B)$ is a linear combination of monomials of strictly
smaller weight.  In the ordinary theory, Letzter and Kolb developed a
projection technique that extracts these lower-weight terms from the
coproduct \cite{Let02,Kol14}. We adapt this mechanism to quantum
supersymmetric pairs and use it systematically. For a Serre polynomial of
weight $\lambda$, the resulting formula is
\[
 p(\underline B)
 =-(\id\otimes\epsilon)
   \bigl(\id\otimes(P_{-\lambda}\circ\pi_{0,0})\bigr)
   \bigl(\Delta(p(\underline B))
          -p(\underline B)\otimes K_{-\lambda}\bigr).
\]
Only certain terms in the coproduct can survive this projection, which
turns the presentation problem into a finite local calculation.

Our first result, \cref{presentation}, is a presentation theorem valid for
quantum supersymmetric pairs of arbitrary basic types. Starting with the superalgebra
freely generated over $\Ub^+\Uio$ by the symbols $\tB_i$ and parities $|\tB_i|=|i|$, we show that
the kernel of the canonical map to $\Ui$ is generated by elements \cref{Serre1}, \cref{Serre2} and
\[
 p(\tB_{i_1},\tB_{i_2},\ldots, \tB_{i_k})
   -\widetilde{\Rem}^{p}_{i_1,i_2,\ldots,i_k}
\]
for every quantum Serre polynomial $p$. Each remainder has strictly lower
weight, and its formal expression is independent of the admissible
parameters $\kappa_i$. Thus an explicit Serre presentation reduces exactly
to determining these remainders.

We determine these remainders for all super Satake diagrams of type~A.
These diagrams comprise the families with $\tau=\id$, including the
orthosymplectic fixed-point case, and the type~sAIII families with
$\tau\ne\id$. For any choice of parities on the type~A Dynkin diagram,
there are four relevant quantum Serre polynomials: the square $p_0(x_i)=x_i^2$ at an
isotropic odd root; the supercommutator $p_1(x_i,x_j)=[x_i,x_j]$ for
nonadjacent nodes; the usual cubic polynomial $p_2(x_i,x_j)$ at an even node $i$;
and the degree-four polynomial
\[
 p_3(x_i,x_j,x_k)
   =[[[x_i,x_j]_{q_j},x_k]_{q_k},x_j]
 \]
when $j$ is isotropic odd and $i\sim j\sim k$ with $i<k$. The remainders for $p_0$, $p_1$, and $p_2$ are given in
\cref{prop:remp0,prop:remp1,prop:remp2}. These relations extend the familiar formulas in the purely even setting while accounting for the signs and isotropic roots arising in the super setting.

The principal new calculation is the degree-four relation in
\cref{prop:remp3}. We classify the possible three-node local Satake
diagrams around an isotropic odd root and compute the remainder in every
case. It vanishes whenever the three-node string has no relevant overlap
with its $\tau$-image, and the nonzero cases fall into two families. For type~sAIII, the correction is a linear combination of
lower-degree $q$-commutators multiplied by the coideal parameters and toral
elements.  For diagrams with black boundary nodes, it is a signed scalar
multiple of a neighboring coideal generator times a toral factor. These
relations have no purely even analogue. Together with the preceding three
families, they yield a complete Serre-type presentation for all type~A
$\imath$quantum enveloping superalgebras.

The paper is organized as follows. In \cref{sec:QSP} we review type~A quantum
enveloping superalgebras, classify the relevant super Satake diagrams, and
recall the construction and basis theorem for $\Ui$. In \cref{sec:iQSP} we prove
the general presentation theorem, develop the projection formula, and
determine successively the remainders for $p_0,p_1,p_2$, and $p_3$.

\subsection*{Acknowledgment}
We would like to thank Weinan Zhang for helpful discussions. YS is partially supported by the Fields Institute and the Science and Technology Commission of Shanghai Municipality (No. 22DZ2229014).

\section{Quantum supersymmetric pairs of type A}
\label{sec:QSP}
In this section, we review the construction of quantum supersymmetric pairs of type A. Throughout this section, we work over the field $\kk = \C(q)$, where $q$ is an indeterminate. When working with superspaces, we denote the parity of a homogeneous element $v$ by $|v| \in \Z_2$. When we write expressions involving $|v|$, we implicitly assume that $v$ is homogeneous. For an element $a\in \Z_2$, we define 
\[
    (-1)^{a}
    =
    \begin{cases}
        1 & \text{if $a=\overline{0}$}, \\
        -1 & \text{if $a=\overline{1}$}.
    \end{cases}
\]
For a statement $P$, we define
\[
    \delta_P
    :=
    \begin{cases}
        1 & \text{if $P$ is true}, \\
        0 & \text{if $P$ is false}.
    \end{cases}
\]
In particular, we have the usual Kronecker delta $\delta_{ij} := \delta_{i=j}$.   

\subsection{Quantum enveloping superalgebras}
Let $V = V_{\overline{0}} \oplus V_{\overline{1}}$ be a vector superspace over $\C$ with $\dim V_{\overline{0}} = m$, $\dim V_{\overline{1}} = n$.  We fix a homogeneous basis of $V$
\[
    v_i,\qquad i \in \Vset := \left\{1,2,\ldots,m+n \right\}.
\]
Equipped with the supercommutator, the superspace $\End(V)$ becomes the general linear Lie superalgebra, which is denoted by $\fg = \fgl(V) = \fgl(m|n,\C)$.  The Lie superalgebra $\fg$ has a homogeneous basis given by the matrix units $E_{ij}$, $i,j \in \Vset$, where the parity of $E_{ij}$ is 
\[
    |E_{ij}| = |v_i| + |v_j|.
\]
Let $\fh$ be the Cartan subalgebra of $\fg$ consisting of diagonal matrices. Then the dual space $\fh^*$ has a basis $\epsilon_i$, $i \in \Vset$, defined by $\epsilon_i(E_{jj}) = \delta_{ij}$ for $j \in \Vset$. We define a symmetric bilinear form on the weight lattice $P = \oplus_{i \in \Vset} \Z\epsilon_i$ by
\[
    (\epsilon_i, \epsilon_j) =
    (-1)^{|v_i|}\delta_{ij}.
\]
We often write $|\epsilon_i|:=|v_i|$ for any $i\in \Vset$.

Now we set
\begin{equation} \label{Idef}
    I := \left\{ 1,2,\ldots, m+n-1 \right\}
\end{equation}
and denote the set of simple roots by
\begin{equation} \label{simpleroots}
    \Pi= \left\{ \alpha_i= \epsilon_{i}-\epsilon_{i+1} \mid i \in I \right\}.
\end{equation}
For any $i\in I$, we define the parity of $\alpha_i$ by $|\alpha_i|=|\epsilon_i|+|\epsilon_{i+1}|$, and we set $|i| :=|\alpha_i|$. The Dynkin diagram associated to this choice is of the form
\begin{equation} \label{dynkin}
    \begin{tikzpicture}[centerzero,semithick]
        \node (1) [circslash, 
            label={[font=\scriptsize]below:$1$}] at (0,0) {};
        \node (2) [circslash, 
            label={[font=\scriptsize]below:$2$}] at (1.5,0){};
        \node (3) at (3,0) {$\cdots$} ;
        \node (4) [circslash, 
            label={[font=\scriptsize]below:$m+n-2$}] at (4.5,0){};
        \node (5) [circslash, 
            label={[font=\scriptsize]below:$m+n-1$}] at (6,0){};
        \path (1) edge (2)
              (2) edge (3)
              (3) edge (4)
              (4) edge (5);
    \end{tikzpicture}
    \ , \quad \text{where} ~
    \begin{tikzpicture}[centerzero,semithick]
        \node [circslash, 
            label={[font=\scriptsize]below:$i$}] at (0,0) {};
    \end{tikzpicture}
    =
    \begin{cases}
        \begin{tikzpicture}[centerzero,semithick]
            \node [circ, 
                label={[font=\scriptsize]below:$i$}] at (0,0) {};
        \end{tikzpicture}
        & \text{if } |\alpha_i| = \overline{0},
        \\
        \begin{tikzpicture}[centerzero,semithick]
            \node [circcross, 
                label={[font=\scriptsize]below:$i$}] at (0,0) {};
        \end{tikzpicture}
        & \text{if } |\alpha_i| = \overline{1}.
    \end{cases}
\end{equation}

We define the coweight lattice $P^\vee = \oplus_{i\in \Vset} \Z \epsilon^\vee_i$, and we have the $\Z$-bilinear pairing
\[
    \langle \cdot,\cdot \rangle \colon P^\vee \times P\to \Z, \qquad
    \langle \epsilon^\vee_i,\epsilon_j \rangle = \delta_{ij}.
\]
The set of simple coroots is given by
\begin{equation} \label{hdef}
    \Pi^\vee = \{ h_i  \mid i \in I \},\qquad \text{where} ~
	h_i := \epsilon_{i}^\vee - (-1)^{|\alpha_i|} \epsilon_{i+1}^\vee.
\end{equation}
We let
\begin{equation} \label{vaddad}
    X := \Z \Pi
    \qquad \text{and} \qquad
    Y := \Z \Pi^\vee
\end{equation}
denote the root lattice and coroot lattice, respectively. The generalized Cartan matrix associated to the above data is
\begin{equation}
\label{Cmatrix}
    A=(a_{ij})_{i,j\in I}, \qquad \text{where} ~
    a_{ij} = \langle h_i, \alpha_j \rangle
    = \left( 1 + (-1)^{|\alpha_i|} \right) \delta_{ij} - \delta_{i,j+1} - (-1)^{|\alpha_i|} \delta_{i,j-1}.
\end{equation}
Note that $A$ is symmetrizable, since
\begin{align} \label{eq:di}
    (\alpha_i, \alpha_j) = d_i a_{ij},\qquad \text{where} ~
    d_i := (-1)^{|\epsilon_i|}, \qquad
    \forall i,j \in I.
\end{align}

We also introduce the following notation. We say that two distinct elements $i, j\in \I$ are {\em connected} if $a_{ij}\neq 0$ and write $i\sim j$. Likewise, we say that two distinct elements $i, j\in \I$ are {\em not connected} if $a_{ij}=0$ and write $i\nsim j$.

Set
\[
    q_i := q^{d_i}, \qquad \text{for} ~i \in I.
\]
Define the quantum integers and quantum binomial coefficients by
\[
    [a]_{q_{i}} := \frac{q_{i}^a-q_{i}^{-a}}{q_{i}-q_{i}^{-1}},
    \qquad
    \qbinom{a}{k} := \frac{[a]_{q_{i}} [a-1]_{q_{i}} \cdots [a-k+1]_{q_{i}}}{[k]_{q_{i}}[k-1]_{q_{i}}\cdots [1]_{q_{i}}},
    \qquad \text{for} ~a \in \Z,\ k \in \Z_{\geq 0}.
\] 

Following \cite[Th.~10.5.1]{Yam94}, we define the quantum enveloping superalgebra $\Ualg = \bU_q(\fg)$ to be the unital associative superalgebra over $\C(q)$ with generators
\[
    E_i, F_i, \quad \text{for} ~i \in I, \qquad
    K_h,\quad \text{for} ~h \in P^\vee,
\]
and parities
\[
    |E_i| = |F_i| = |\alpha_i|,\quad \text{for} ~i \in I, \qquad
    |K_h| = 0,\quad \text{for} ~h \in P^\vee,
\]
subject to the following relations for $h,h' \in P^\vee$, $i,j \in I$:
\begin{gather}
    K_0 = 1,\qquad
    K_{h} K_{h'} = K_{h+h'},
	\\ \label{KEF}
    K_h E_i = q^{\langle h, \alpha_i \rangle} E_i K_h,\qquad
    K_h F_i = q^{- \langle h, \alpha_i \rangle} F_i K_h,
    \\ \label{skunk}
    [E_i,F_j]
	= \delta_{ij} \frac{K_i - K_i^{-1}}{q_i-q_i^{-1}},\qquad
	\text{where } K_i := K_{h_i}^{d_i},
    \\ \label{square}
    [E_i,E_j] = 0,\qquad
	[F_i,F_j] = 0,\qquad
	\text{if } a_{ij}=0,
    \\ \label{racoon}
    [E_i,[E_i,E_j]_{q_i^{-1}}]_{q_i}= 0 = [F_i,[F_i,F_j]_{q_i^{-1}}]_{q_i},\qquad
	\text{if } |a_{ij}|=1,\ |\alpha_i|=\overline{0},
    \\ \label{Serre}
        [[[E_{j-1},E_j]_{q_j},E_{j+1}]_{q_{j+1}}],E_j]=0,\quad
        [[[F_{j-1},F_j]_{q_j},F_{j+1}]_{q_{j+1}}],F_j]=0,
        \text{ if }  1<j<m+n-1,\ |\alpha_j|=\overline{1},
\end{gather}
where
\[
    [x,y]_a := xy-(-1)^{|x| |y|} a yx, \quad \text{for} ~a \in \C(q),
    \qquad \text{and} \qquad [x,y] := [x,y]_1.
\]
Then $\Ualg$ is a Hopf superalgebra with the coproduct $\Delta$, counit $\varepsilon$, and antipode $S$ given as follows: for $i \in I$ and $h \in P^\vee$,
\begin{gather} \label{comult}
	\Delta(E_i) = E_i \otimes 1 + K_i \otimes E_i,\quad
	\Delta(F_i) = F_i \otimes K_i^{-1} + 1 \otimes F_i,\quad
	\Delta(K_h) = K_h \otimes K_h,
	\\ \label{counit}
	\varepsilon(E_i) = 0,\qquad
	\varepsilon(F_i) = 0,\qquad
	\varepsilon(K_h) = 1,
	\\ \label{antipode}
	S(E_i) = -K_i^{-1} E_i,\qquad
	S(F_i) = -F_i K_i,\qquad
	S(K_h) = K_{-h}.
\end{gather}

For each $i \in I$ such that $|\alpha_i|=0$, we have an automorphism of superalgebras $T_i \colon \Ualg \to \Ualg$ given by
\begin{equation} \label{Tidef}
	\begin{gathered}
		T_i (E_i) = -F_i K_{i},\qquad
		T_i (F_i) = - K_{i}^{-1}E_i,
		\\
		T_i(E_j) = E_j,\qquad T_i(F_j)=F_j, \qquad \text{for} ~ |j-i|>1,
		\\
		T_i(E_j) = E_i E_j - q^{(\alpha_i,\alpha_j)} E_j E_i,\qquad
		T_i(F_j) = F_i F_j - q^{-(\alpha_i,\alpha_j)} F_j F_i,
		\qquad \text{for} ~ |j-i|=1,
		\\
		T_i (K_h) = K_{s_i(h)},\quad \text{for} ~h \in P^\vee,
	\end{gathered}
\end{equation}
where $s_i$ is the reflection corresponding to the simple root $\alpha_i$. The $T_i$'s are super versions of the automorphisms $T_{i,1}''$ of \cite[\S 37.1.3]{Lus10} and satisfy braid group relations; see \cite[Prop.~7.4.1]{Yam99} and \cite[Th.~3.4]{She25}. Let $\psi \colon \Ualg\to \Ualg$ denote the $\C$-algebra automorphism, often called the \emph{bar involution}, determined by
\begin{equation} \label{Ualgbar}
    \psi(E_i) = E_i,\quad 
    \psi(F_i) = F_i,\quad
    \psi(K_{h}) = K_{-h},\quad 
    \psi(q) = q^{-1}.
\end{equation}

As in \cite[Lem. 2.1.4]{Yam94}, the multiplication gives a vector superspace isomorphism (a triangular decomposition of $\Ualg$)
\begin{equation} \label{eq:triangle}
	\Ualg^+ \otimes \Ualg^0 \otimes \Ualg^- \cong \Ualg,
\end{equation}
where  $\Ualg^{+}$ (resp. $\Ualg^{-}$) denotes the subalgebra of $\Ualg$ generated by $E_j $ (resp. $F_j $), $j \in I$, and $\Ualg^{0}$ denotes the subalgebra of $\Ualg$ generated by $K_\mu$, $\mu \in P^\vee$. Let $\C(q)[X]$ be the group algebra of the root lattice $X$. There exists an inclusion of algebras $\C(q)[X]\hookrightarrow {\Ualg^0}$ such that $\alpha_i\mapsto K_i$, and hence we may also write
\[K_\beta:=\prod_{i\in \I}K_i^{n_i},
\qquad
\text{for } \beta=\sum_{i\in \I}n_i\alpha_i \in X.
\]
As in \cite[Sect. 4.3]{SW25}, the triangular decomposition \cref{eq:triangle} induces an isomorphism of vector spaces
\begin{equation}
    \Ualg^+\otimes \Ualg^0\otimes S(\Ualg^-)\cong \Ualg.
\end{equation}
This leads to a direct sum decomposition 
\begin{equation}\label{eq:projectionP}
    \Ualg=\bigoplus_{h \in P^\vee}\Ualg^+K_h S(\Ualg^-).
\end{equation}
For any $h\in P^\vee$, let 
\begin{gather}
\label{eq:Pprojection}
P_h:\Ualg \longrightarrow \Ualg^+K_hS(\Ualg^-)
\end{gather}
denote the projection with respect to \cref{eq:projectionP}. We also use the symbol $P_\lambda$ for $\lambda\in X$ to denote the projection $P_\lambda:\Ualg \to \Ualg^+K_{\lambda} S(\Ualg^-)$ as above.
We have that 
\begin{equation}
\label{eq:deltaP}
    \Delta\circ P_h(x)=(\id\otimes P_h)\Delta(x),\ \forall h\in P^\vee, x\in \Ualg.
\end{equation}

On the other hand, let $X^+:=\Z_{\geq 0}\Pi$. We denote
 \[ 
 \Ualg^\pm_\beta:=\{u\in \Ualg^\pm\mid K_iuK_i^{-1}=q^{(\alpha_i,\beta)}u,\ \forall i\in \I\},
 \qquad \text{ for any } \beta\in X.
 \]
Then, with respect to the decomposition 
\begin{equation}
    \Ualg=\bigoplus_{\alpha,\beta\in X^+}\Ualg_\alpha^+\Ualg^0\Ualg_{-\beta}^-,
\end{equation}
we have the following projections
\begin{align}
    \label{eq:piprojection}
    \pi_{\alpha,\beta}:\Ualg\longrightarrow \Ualg_\alpha^+\Ualg^0\Ualg_{-\beta}^-.
\end{align} 
Moreover, for any $\mu,\gamma\in \fh^*$, we write $\mu>\gamma$ if $\mu-\gamma\in X^+$ and $\mu\neq \gamma$.

\subsection{Super Satake diagrams of type A}
 In this subsection, we review the notion of super Satake diagrams given in \cite[Sect. 2.2]{SW25}. Let \(\I=\Ieven \sqcup \Iodd\), where $\Ieven$ (resp. $\Iodd$) consists of $i\in \I$ such that $|i|=\overline{0}$ (resp. $|i|=\overline{1}$). 
 
For a subset $\I_\bu\subset \Ieven$, we denote by \(\fg_\bu\) the semisimple Lie subalgebra of \(\fg\) associated with \(\I_\bu\), and by $\Phi_{\bu}\subset \Phi$ the corresponding root system, where $\Phi$ is the root system associated to the Cartan subalgebra $\fh$. Denote by \(W_\bu\) the Weyl group of \(\fg_\bu\) with the longest element \(w_\bu\). We regard \(\fg_\bu\) as the semisimple part of the corresponding Levi subalgebra of \(\fg\), and therefore the positive coroots for \(\fg_\bu\) can be regarded as part of coroots for \(\fg\). Let $\rho_\bu^\vee$ denote half the sum of the positive coroots in $\Phi_{\bu}$. 

\begin{defin} (\cite[Def. 2.3 and Prop. 2.10]{SW25})
\label{def:superad}
A pair $(\I=\I_\circ \sqcup\I_\bu,\tau)$ consisting of subsets $\I_\bu\subset \Ieven$, $I_\circ=\I\backslash \I_\bu$, and a parity-preserving permutation $\tau$ of $\I$ is called {\em super admissible} if the following conditions are satisfied:
\begin{enumerate}
    \item[(1)] $d_i a_{ij}=d_{\tau i} a_{\tau i, \tau j}$, $\forall i,j \in I$, and $\tau^2=\id$,
    \item[(2)] 
     $\tau(\I_\bu)=\I_\bu$, and the action of $\tau$ on $\I_\bu$ coincides with the action of $-w_\bu$,
    \item[(3)]
    If $j\in \Ieven\cap \I_\circ$ and $\tau(j)=j$, then $\langle \rho_\bu^\vee,\alpha_j\rangle\in \Z$, 
    \item[(4)]
    If $j\in  \Iodd\cap\I_\circ$ and $\tau(j)=j$, then
$\langle 2\rho_\bu^\vee,\alpha_j\rangle\equiv 1 ~(\mathrm{ mod }~2)$.
\end{enumerate}
\end{defin}
\begin{rem}
Our \cref{def:superad} differs slightly from \cite[Def.~2.3]{SW25}.  
Indeed, since we only consider type~A Lie superalgebras, non-isotropic odd simple roots do not occur in our setting. Moreover, condition (4) ensures that the automorphism $\theta(\I, \tau)$ of $\fg$ defined in \cite[(2.5)]{SW25} has order 2.
\end{rem}
We represent super admissible pairs by diagrams, called \emph{super Satake diagrams}, and we will use the terms ``super admissible pairs'' and ``super Satake diagrams'' interchangeably. In these diagrams, we will use \(\newmoon\) to represent simple even roots in \(\I_\bu\).

Throughout the paper, a subdiagram \(\Gamma\subseteq\Gamma'\) of a super Satake diagram \(\Gamma'\) is understood to be a subgraph satisfying the following condition: whenever \(\Gamma\) contains a white node \(i\in\I_\circ\), it also contains $\tau i$ and every connected component of the black nodes of \(\Gamma'\) that is adjacent to \(i\).

Various examples of super Satake diagrams of general type can be found in \cite{SW25}. In type~A, the diagrams divide into two families according to whether \(\tau=\id\) or not. When $\tau=\id$, all possible super Satake diagrams of type A are explored in detail in \cite[Ex. 2.16]{SW25}. We briefly recall the possibilities here. For this purpose, we introduce the following shorthand notations for local diagrams.
{\rm
\[
{}_l\XBox :=\xy
(-5,0)*{\fullmoon};(0,0)*{\fullmoon}**\dir{-};(5,0)*{\cdots}**\dir{-};(10,0)*{\fullmoon}**\dir{-};(15,0)*{\otimes}**\dir{-};
(-5,-4)*{\scriptstyle l};(0,-4)*{\scriptstyle l-1};(10,-4)*{\scriptstyle 2};(15,-4)*{\scriptstyle 1};
\endxy\ ,
\quad
\XBox_k :=\xy
(-5,0)*{\otimes};(0,0)*{\fullmoon}**\dir{-};(5,0)*{\cdots}**\dir{-};(10,0)*{\fullmoon}**\dir{-};(15,0)*{\fullmoon}**\dir{-};
(-5,-4)*{\scriptstyle 1};(0,-4)*{\scriptstyle 2};(10,-4)*{\scriptstyle k-1};(15,-4)*{\scriptstyle k};
\endxy\ ,
\quad
\davidsstar_t :=\xy
(-5,0)*{\otimes};(0,0)*{\fullmoon}**\dir{-};(5,0)*{\cdots}**\dir{-};(10,0)*{\fullmoon}**\dir{-};(15,0)*{\otimes}**\dir{-};
(0,-4)*{\scriptstyle 1};(10,-4)*{\scriptstyle t-1};
\endxy
\]
}for $l,k,t\geq 1$. For example, we have {\rm$\davidsstar_1= \xy
(0,0)*{\otimes};(5,0)*{\otimes}**\dir{-};
\endxy$} and {\rm$\XBox_1={}_1\XBox= \otimes$}. We also set {\rm$\XBox_0={}_0\XBox=\varnothing$} (which are to be removed) and {\rm$\davidsstar_0= \fullmoon$} (the even split rank one Satake diagram). 

Consider the following diagram
{\rm{
\begin{equation}
    \label{eq:splitA}
\xy
(0,0)*{{}_l\XBox};(8,0)*{\newmoon}**\dir{-};(16,0)*{\davidsstar_{n_1}}**\dir{-};(24,0)*{\newmoon}**\dir{-};(32,0)*{\davidsstar_{n_2}}**\dir{-};(40,0)*{\newmoon}**\dir{-};(47,0)*{\cdots}**\dir{-};(56,0)*{\davidsstar_{n_a}}**\dir{-};(64,0)*{\newmoon}**\dir{-};(72,0)*{\XBox_k}**\dir{-};
\endxy
\end{equation} }}
with the underlying Lie superalgebra \[\fg=\gl(l|2|n_1|2|n_2|2|\cdots|n_a|2|k) \cong \gl(l +n_1 + \cdots + n_a +k|2a+2) \text{ or } \gl(2a+2|l +n_1 + \cdots + n_a +k)\] depending on the parities and $\tau=\id$. Notably, in this case, the fixed-point subalgebra is $$\osp(l +n_1 + \cdots + n_a +k|2a+2)\cong \mathfrak{spo}(2a+2|l +n_1 + \cdots + n_a +k).$$

When \(\tau\neq \id\), the corresponding super Satake diagrams are of type sAIII, which is a particularly significant case and has been extensively studied in the literature; see \cite{She25,SZ26}. In this case, the super Satake diagrams are of the following forms 
\begin{equation} \label{AIIIdiagram}
    \begin{tikzpicture}[anchorbase,scale=1, semithick]
        \node (-4) [circslash, 
            ] at (0,1) {};
        \node (-3)  at (2,1) {$\cdots$};
        \node (-2) [circslash, 
            ] at (4,1) {};
        \node (-1) [circblack, 
            ] at (5,0.7) {};
        \coordinate (0) at (5,0);
        \node (1) [circblack, 
            ] at (5,-0.7) {};
        \node (2) [circslash, 
           ] at (4,-1) {};
		\node (3)  at (2,-1){$\cdots$};
        \node (4) [circslash, 
           ] at (0,-1) {};
        \path (-4) edge (-3)
            (-3) edge (-2)
            (-2) edge (-1)
            (-1) edge (0)
            (0) edge (1)
            (1) edge (2)
            (2) edge (3)
            (3) edge (4);
        \path (-4) edge[dashed,bend right,<->] (4)
            (-2) edge[dashed,bend right,<->] (2)
            (-3) edge[dashed,bend right,<->] (3);
        \node[fill=white, inner xsep=1pt, inner ysep=-3pt] at (0) {$\raisebox{1.3ex}{$\vdots$}$};
	\end{tikzpicture},\qquad
    \begin{tikzpicture}   [anchorbase,scale=1, semithick]     
    \node (-4) [circslash, 
            ] at (0,.5) {};
        \node (-3)  at (2,.5) {$\cdots$};
        \node (-2) [circslash, 
            ] at (4,.5) {};
        \node (0) [circ
            ] at (5,0) {};
        \node (2) [circslash, 
           ] at (4,-.5) {};
		\node (3)  at (2,-.5){$\cdots$};
        \node (4) [circslash, 
           ] at (0,-.5) {};
        \path (-4) edge (-3)
            (-3) edge (-2)
            (-2)  edge (0)
            (0)  edge (2)
            (2) edge (3)
            (3) edge (4);
        \path (-4) edge[dashed,bend right,<->] (4)
            (-2) edge[dashed,bend right,<->] (2)
            (-3) edge[dashed,bend right,<->] (3);
    \end{tikzpicture}
\end{equation}
where
\(
\begin{tikzpicture}[centerzero,semithick]
    \node [circslash] at (0,0) {};
\end{tikzpicture}
\)
represents a simple root in \(\I_\circ\) of either parity, \(\fullmoon\) represents an even simple root in \(\I_\circ\), and \(\newmoon\) represents a simple root in \(\I_\bu\), which is required to be even by \cref{def:superad}. We also note that in the right-hand diagram in \cref{AIIIdiagram}, the \(\tau\)-fixed node must be an even simple root by \cref{def:superad}(3) and (4). To summarize, \cref{eq:splitA} and \cref{AIIIdiagram} together exhaust all possible super Satake diagrams of type~A.
\begin{eg}
	Suppose $m=4$, $n=2$. All possible super Satake diagrams with $\tau \neq \id$ are as follows:
	
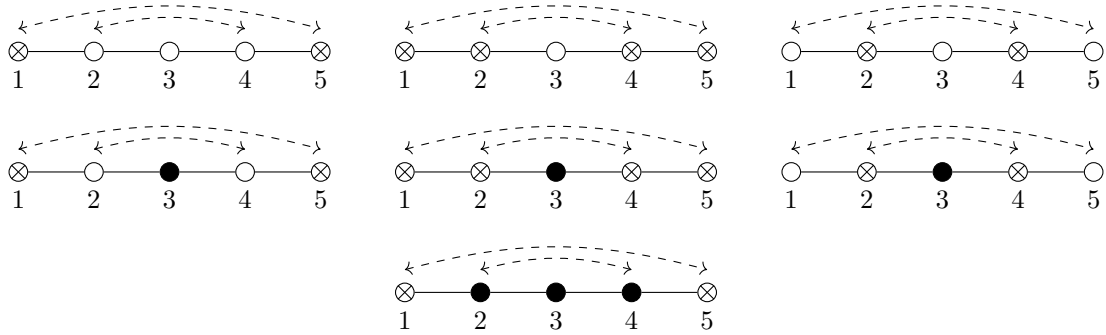
\begin{gather*}
      \begin{tikzpicture}
            \node (-2) [circcross,
                label={below:{$1$}}] at (-2,0) {};
            \node (-1) [circ,
                label={below:{$2$}}] at (-1,0) {};
            \node (0) [circ,
                label={below:{$3$}}] at (0,0) {};
            \node (1) [circ,
                label={below:{$4$}}] at (1,0) {};
            \node (2) [circcross,
                label={below:{$5$}}] at (2,0) {};
            \path (-2) edge (-1)
                (-1) edge (0)
                (0) edge (1)
                (1) edge (2);
            \draw[<->, dashed, bend left=15] (-2, .3) to (2, .3); 
            \draw[<->, dashed, bend left=15] (-1, .3) to (1, .3); 
        \end{tikzpicture}
        \qquad
		\begin{tikzpicture}
            \node (-2) [circcross,
                label={below:{$1$}}] at (-2,0) {};
            \node (-1) [circcross,
                label={below:{$2$}}] at (-1,0) {};
            \node (0) [circ,
                label={below:{$3$}}] at (0,0) {};
            \node (1) [circcross,
                label={below:{$4$}}] at (1,0) {};
            \node (2) [circcross,
                label={below:{$5$}}] at (2,0) {};
            \path (-2) edge (-1)
                (-1) edge (0)
                (0) edge (1)
                (1) edge (2);
            \draw[<->, dashed, bend left=15] (-2, .3) to (2, .3); 
            \draw[<->, dashed, bend left=15] (-1, .3) to (1, .3); 
        \end{tikzpicture}
        \qquad
		\begin{tikzpicture}
            \node (-2) [circ,
                label={below:{$1$}}] at (-2,0) {};
            \node (-1) [circcross,
                label={below:{$2$}}] at (-1,0) {};
            \node (0) [circ,
                label={below:{$3$}}] at (0,0) {};
            \node (1) [circcross,
                label={below:{$4$}}] at (1,0) {};
            \node (2) [circ,
                label={below:{$5$}}] at (2,0) {};
            \path (-2) edge (-1)
                (-1) edge (0)
                (0) edge (1)
                (1) edge (2);
            \draw[<->, dashed, bend left=15] (-2, .3) to (2, .3); 
            \draw[<->, dashed, bend left=15] (-1, .3) to (1, .3); 
        \end{tikzpicture}
    \\
        \begin{tikzpicture}
            \node (-2) [circcross,
                label={below:{$1$}}] at (-2,0) {};
            \node (-1) [circ,
                label={below:{$2$}}] at (-1,0) {};
            \node (0) [circblack,
                label={below:{$3$}}] at (0,0) {};
            \node (1) [circ,
                label={below:{$4$}}] at (1,0) {};
            \node (2) [circcross,
                label={below:{$5$}}] at (2,0) {};
            \path (-2) edge (-1)
                (-1) edge (0)
                (0) edge (1)
                (1) edge (2);
            \draw[<->, dashed, bend left=15] (-2, .3) to (2, .3); 
            \draw[<->, dashed, bend left=15] (-1, .3) to (1, .3); 
        \end{tikzpicture}
		\qquad
		\begin{tikzpicture}
            \node (-2) [circcross,
                label={below:{$1$}}] at (-2,0) {};
            \node (-1) [circcross,
                label={below:{$2$}}] at (-1,0) {};
            \node (0) [circblack,
                label={below:{$3$}}] at (0,0) {};
            \node (1) [circcross,
                label={below:{$4$}}] at (1,0) {};
            \node (2) [circcross,
                label={below:{$5$}}] at (2,0) {};
            \path (-2) edge (-1)
                (-1) edge (0)
                (0) edge (1)
                (1) edge (2);
            \draw[<->, dashed, bend left=15] (-2, .3) to (2, .3); 
            \draw[<->, dashed, bend left=15] (-1, .3) to (1, .3); 
        \end{tikzpicture}
        \qquad
		\begin{tikzpicture}
            \node (-2) [circ,
                label={below:{$1$}}] at (-2,0) {};
            \node (-1) [circcross,
                label={below:{$2$}}] at (-1,0) {};
            \node (0) [circblack,
                label={below:{$3$}}] at (0,0) {};
            \node (1) [circcross,
                label={below:{$4$}}] at (1,0) {};
            \node (2) [circ,
                label={below:{$5$}}] at (2,0) {};
            \path (-2) edge (-1)
                (-1) edge (0)
                (0) edge (1)
                (1) edge (2);
            \draw[<->, dashed, bend left=15] (-2, .3) to (2, .3); 
            \draw[<->, dashed, bend left=15] (-1, .3) to (1, .3); 
        \end{tikzpicture}
        \\
           \begin{tikzpicture}
            \node (-2) [circcross,
                label={below:{$1$}}] at (-2,0) {};
            \node (-1) [circblack,
                label={below:{$2$}}] at (-1,0) {};
            \node (0) [circblack,
                label={below:{$3$}}] at (0,0) {};
            \node (1) [circblack,
                label={below:{$4$}}] at (1,0) {};
            \node (2) [circcross,
                label={below:{$5$}}] at (2,0) {};
            \path (-2) edge (-1)
                (-1) edge (0)
                (0) edge (1)
                (1) edge (2);
            \draw[<->, dashed, bend left=15] (-2, .3) to (2, .3); 
            \draw[<->, dashed, bend left=15] (-1, .3) to (1, .3); 
        \end{tikzpicture}
	\end{gather*}
    where $I_\bu=\varnothing,\{3\},\{2,3,4\}$ for these diagrams in the first, second, and third rows, respectively.
    
    Suppose $m=4$ and $n=1$. Then all possible super Satake diagrams with $\tau=\id$ are as follows:
    
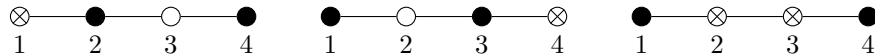
\begin{gather*}
         \begin{tikzpicture}
            \node (-2) [circcross,
                label={below:{$1$}}] at (-2,0) {};
            \node (-1) [circblack,
                label={below:{$2$}}] at (-1,0) {};
            \node (0) [circ,
                label={below:{$3$}}] at (0,0) {};
            \node (1) [circblack,
                label={below:{$4$}}] at (1,0) {};
            \path (-2) edge (-1)
                (-1) edge (0)
                (0) edge (1);
        \end{tikzpicture}
        \qquad
            \begin{tikzpicture}
            \node (-2) [circblack,
                label={below:{$1$}}] at (-2,0) {};
            \node (-1) [circ,
                label={below:{$2$}}] at (-1,0) {};
            \node (0) [circblack,
                label={below:{$3$}}] at (0,0) {};
            \node (1) [circcross,
                label={below:{$4$}}] at (1,0) {};
            \path (-2) edge (-1)
                (-1) edge (0)
                (0) edge (1);
        \end{tikzpicture}
        \qquad
          \begin{tikzpicture}
            \node (-2) [circblack,
                label={below:{$1$}}] at (-2,0) {};
            \node (-1) [circcross,
                label={below:{$2$}}] at (-1,0) {};
            \node (0) [circcross,
                label={below:{$3$}}] at (0,0) {};
            \node (1) [circblack,
                label={below:{$4$}}] at (1,0) {};
            \path (-2) edge (-1)
                (-1) edge (0)
                (0) edge (1);
        \end{tikzpicture}
    \end{gather*}
    where $I_\bu=\{2,4\},\{1,3\},\{1,4\}$ for the above three diagrams, respectively.
\end{eg}

\subsection{$\imath$Quantum enveloping superalgebras}
In this subsection, we review the construction of the $\imath$quantum enveloping superalgebra, also known as the $\imath$quantum supergroup, associated with a given super Satake diagram
\(
(\I=\I_\circ\sqcup \I_\bu,\tau),
\)
following \cite[Sect.~4.1]{SW25}. The diagram involution $\tau$ induces an automorphism of $\fg$ given by
\begin{align}
\label{eq:tau}
\begin{split}
\tau \colon\fg & \longrightarrow\fg, 
\\
e_i &\mapsto \begin{cases}
    e_{i} & \text{if }\tau i=i, \\
    (-1)^{|i|}e_{\tau i}&\text{if } \tau i\neq i,
\end{cases}
\qquad 
f_i\mapsto f_{\tau i},\qquad 
h_i\mapsto \begin{cases}
    h_{i} &\text{if } \tau i=i, \\
    (-1)^{|i|}h_{\tau i}&\text{if } \tau i\neq i,
\end{cases}
\end{split}
\end{align}
where $e_i, f_i$, for $i \in I$, are Chevalley generators of $\fg$. It also induces an operator \(\tau\) acting on \(X\) by setting
\[
\tau(\alpha_i)=\alpha_{\tau i},\qquad \text{for} ~i\in \I,
\]
and extending it linearly.
 
Recall $Y$ defined in \eqref{vaddad}. We define
\[
Y^\io=\{h\in Y\mid -w_\bu\circ\tau(h)=h\}.
\]
Let $\Uio$ denote the subalgebra generated by $K_h$ for all $ h\in Y^\io$. For an element $w\in W_\bu$ and each reduced expression $w=s_{i_1}\cdots s_{i_k}$, we may define $T_{w}=T_{i_1}\cdots T_{i_k}$, which is well-defined since the $T_i$'s, for $i\in \Ieven$, satisfy braid relations. Recall that \(w_\bu\) is the longest element of \(W_\bu\).

\begin{defin}(\cite[Def. 4.1]{SW25})
\label{def:Ui}
    The algebra $\Ui$, with parameters $\va_i \in \C(q)^*, \kappa_i\in \C(q), \text{ for }i\in \I_\circ$,
    is the $\C(q)$-subalgebra of $\Ualg$ generated by $\Uio$ together with the following elements:
    \begin{align*}
        B_i=F_i+\va_i T_{w_\bu}(E_{\tau i})K_i^{-1}+\kappa_i K_i^{-1},\ \forall i\in \I_\circ, \qquad
        E_i,\ F_i, \  \forall i\in \I_\bu.
    \end{align*}
\end{defin}
We extend the definition of $B_i$ by setting $B_i=F_i$ for any $i\in I_\bu$. It is shown in \cite[Prop. 4.2]{SW25} that $\Ui$ is a right coideal subalgebra of $\Ualg$, that is, 
\[
\Delta(\Ui)\subset \Ui\otimes \Ualg.
\]

Define 
\begin{equation}
\label{eq:Ins}
    \I_{\ns}=\{i\in \I_\circ\mid \tau i=i,\ \langle h_j,\alpha_i\rangle=0,\ \forall j\in \I_\bu\}.
\end{equation}
Then, in order to ensure the crucial property that $\Ui\cap \Ualg^0=\Uio$, the parameters of $\Ui$ are supposed to satisfy the following conditions (cf. \cite[Sect. 4.2]{SW25}):
\begin{enumerate}
    \item  \(\kappa_i=0 \text{ unless } i\in \I_{\ns} \text{ and } \langle h_k,\alpha_i\rangle\in 2\Z \text{ for all } k\in (\I_{\ns}\backslash\{i\})\cup \Iodd\),
    \item \(\va_i=\va_{\tau i}\) if \((\alpha_i,\alpha_{\tau i})=0\) and \(w_\bu(\alpha_i)=\alpha_i\),
    \item  \(\va_i=\va_{\tau i}\) if $|i|=\overline{1}$ and $(\alpha_i,\alpha_{\tau i})\in 2\Z\backslash \{0\}$.
\end{enumerate}
\begin{rem}
Note that we have strengthened the condition in \cite[(4.13)]{SW25} to eliminate certain extreme cases that it fails to account for. In our type~A setting, condition~(c) on the parameters does not arise.  For super Satake diagrams of the form \cref{eq:splitA}, we have \(\kappa_i=0\) for all \(i\in I_\circ\), except in the case where the Satake diagram consists of a single node \(\fullmoon\).  
For super Satake diagrams of the form \cref{AIIIdiagram}, we likewise have \(\kappa_i=0\), except possibly when \(i\) is the unique \(\tau\)-fixed node. As we shall see later, the presentation of \(\Ui\) is essentially independent of the parameters \(\kappa_i\). Similar observations have been made in the purely even case; see \cite{Kol14,Wat21}.
\end{rem}
By \cite[Lem. 4.3]{SW25}, for any $i\in I_\bu,\ k\in I$, we have
\begin{equation}
\label{eq:EjBk}
E_iB_k-(-1)^{|i||k|}B_kE_i=\delta_{i,k}\frac{K_i-K_i^{-1}}{q_i-q_i^{-1}}.
\end{equation}

Denote by $\Ualg_{\bu}= \Ualg^+_{\bu}\Ualg^0_{\bu}\Ualg^-_{\bu}$ the quantum group associated with $\I_\bu\subset \Ieven$, where  $\Ualg^+_{\bu},\ \Ualg^-_{\bu}$ and $\Ualg^0_{\bu}$ denote the subalgebras generated by $\{E_i\mid i\in \I_\bu\},\ \{F_i\mid i\in \I_\bu\}$ and $\{K_i^{\pm 1}\mid i\in \I_\bu\}$, respectively. Define a filtration $\Fi^*$ on $\Ui$ by the following degree function specified on generators:
\begin{equation}
    \label{eq:degree}
    \deg (B_i)=1,\ \forall i\in \I,\quad \deg(x)=0,\  \forall x\in \Ub^+\Uio.
\end{equation}
For any $J=(j_1,\ldots,j_r)\in I^r$, we denote $\wt(J)=\sum_{i=1}^r\alpha_{j_i}$ and introduce the following shorthand notations:
\begin{equation*}
E_J=E_{j_1}\cdots E_{j_r},\quad F_J=F_{j_1}\cdots F_{j_r},\quad B_J=B_{j_1}\cdots B_{j_r}.
\end{equation*}
Let $\mathcal{J}$ be a fixed subset of $\cup_{s\in \Z_{\geqslant 0}}I^s$ such that $\{F_J\mid J\in \mathcal{J}\}$ forms a basis for $\Ualg^{-}$. We recall the following basis theorem from \cite[Thm. 4.14]{SW25}.
\begin{theo}
\label{basisthm}
The set $\{B_J \mid J \in \mathcal{J}\}$ forms a basis for the left $\Ualg^{+}_{\bu}\Uio$-module $\Ui$.
\end{theo}

\section{A Serre presentation of $\imath$quantum enveloping superalgebras}
\label{sec:iQSP}
\subsection{Quantum Serre polynomials}
For an ordinary quantum group, the Serre relations can be written as $S_{ij}(F_i,F_j) =S_{ij}(E_i,E_j)=0$ in terms of a non-commutative Serre polynomial
\[
S_{ij}(x,y)=\sum_{k=0}^{1-a_{ij}}(-1)^k\qbinom{1-a_{ij}}{k}_{q_i}x^{1-a_{ij}-k}yx^k.
\]
However, according to \cite{Yam94,Yam99}, various higher-order quantum Serre relations will appear for quantum supergroups associated with arbitrary Dynkin diagrams; see \cite[Table~4]{SW25} for the full list of local Dynkin diagrams related to the super admissible pairs and the corresponding Serre relations. In our type~A setting, only the cases (ISO1), (ISO2), (N-ISO), and (AB) from that table are needed. We list all the quantum Serre polynomials relevant to type~A below:
\begin{equation}
\label{Serrepoly}
\begin{gathered}
        (i\in \Iodd) \leadsto p_0(x_i)=x_i^2,\qquad (i\nsim j\in \I)  \leadsto p_1(x_i,x_j)=[x_i,x_j], \\ (i\in \Ieven, i\sim j) \leadsto p_2(x_i,x_j)=[x_i,[x_i,x_j]_{q_i^{-1}}]_{q_i}, \\
    (j\in \Iodd, i\sim j\sim k ~\mathrm{ with }~i<k) \leadsto p_3(x_i,x_j,x_k)=[[[x_i,x_j]_{q_j},x_k]_{q_{k}},x_j], 
\end{gathered}
\end{equation}
where the parity of $x_i$ is $|x_i|:=|i|$ for any $i\in \I$. Here we use the \(q\)-commutator notation for the degree-four quantum Serre polynomial, as in \cite{Yam94}. In \cite{SW25}, the corresponding quantum Serre relations are written in terms of explicit noncommutative polynomials. These two formulations are equivalent: after expanding the iterated \(q\)-commutators and then substituting the variables with the Chevalley generators of $\Ualg$, one obtains exactly the same polynomial relations as those appearing in \cite{SW25} due to the fact that $F_j^2=0=E_j^2$ whenever $j\in \Iodd$. 

\begin{eg}
    Consider the following super Satake diagram:
    
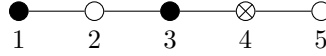
\[
     \begin{tikzpicture}
            \node (-2) [circblack,
                label={below:{$1$}}] at (-2,0) {};
            \node (-1) [circ,
                label={below:{$2$}}] at (-1,0) {};
            \node (0) [circblack,
                label={below:{$3$}}] at (0,0) {};
            \node (1) [circcross,
                label={below:{$4$}}] at (1,0) {};
            \node (2) [circ,
                label={below:{$5$}}] at (2,0) {};
            \path (-2) edge (-1)
                (-1) edge (0)
                (0) edge (1)
                (1) edge (2);
        \end{tikzpicture}
    \]
    We list several quantum Serre polynomials associated with the underlying Dynkin diagram.
    Since $|\alpha_4|=\overline{1}$, we see that
    \[
    p_0(x_4)=x_4^2,\quad p_3(x_3,x_4,x_5)=[[[x_3,x_4]_{q_4},x_5]_{q_{5}},x_4]
    \]
    are two possible quantum Serre polynomials associated with this diagram. Since $|\alpha_2|=|\alpha_5|=\overline{0}$, we also have 
    \[
    p_2(x_2,x_3)=[x_2,[x_2,x_3]_{q_2^{-1}}]_{q_2},\qquad p_2(x_5,x_4)=[x_5,[x_5,x_4]_{q_5^{-1}}]_{q_5}.
    \]
\end{eg}

\subsection{The projection technique}
The projection technique was introduced by Letzter \cite{Let02} and further developed by Kolb \cite{Kol14} for purely even quantum symmetric pairs; later on, it was extended to the super setting in \cite{SW25}. It provides a powerful method for determining the Serre presentation of \(\Ui\), as in \cite[Sect. 5 and 7]{Kol14}. More generally, this technique provides an effective way to compare \(\Ui\) with the quantum supergroup $\Ualg$, and it plays an important role in formulating results such as the quantum Iwasawa decomposition from \cite[Thm. 4.15]{SW25}.

Let $p=p(x_{i_1},\ldots,x_{i_\ell})$ be a quantum Serre polynomial in variables $x_{i_t}$ for $i_t\in \I$. To simplify notation we write
\begin{gather*}
    p(\underline{F})=p(F_{i_1},\ldots,F_{i_\ell}),\quad p(\underline{E})=p(E_{i_1},\ldots,E_{i_\ell}),\quad p(\underline{B})=p(B_{i_1},\ldots,B_{i_\ell}),\\
    p\big(\underline{T_{w_\bu}(E_\tau)K^{-1}}\big) =p(T_{w_\bu}(E_{\tau i_1})K_{i_1}^{-1},\ldots, T_{w_\bu}(E_{\tau i_\ell})K_{i_\ell}^{-1}).
\end{gather*}
Since \(p\) is homogeneous, it has a well-defined weight, which we write as
\[
\lambda=\lambda_{1}\alpha_{i_1}+\cdots+\lambda_{\ell}\alpha_{i_\ell},
\quad \text{where} ~ \lambda_t\in \mathbb \Z_{\geq 0}.
\]

We collect several results that will be useful in determining the Serre presentation of \(\Ui\). Recall the projections \cref{eq:Pprojection} and \cref{eq:piprojection}.
\begin{lem}
\label{Proj=0}
    Let $p$ be a quantum Serre polynomial of weight $\lambda$. Then we have $P_{-\lambda}(p(\underline{B}))=0$.
\end{lem}
\begin{proof}
    This is a direct consequence of \cite[Prop. 4.8 and 4.11]{SW25}.
\end{proof}

\begin{lem} (cf. \cite[Cor. 5.17]{Kol14})
\label{lem:remain}
       Let $p$ be a quantum Serre polynomial of weight $\lambda$. Then we have
       \begin{gather}
       \label{remain}
           p(\underline{B})\in \sum_{\{J\in\mathcal{J}|\wt(J)<\lambda\}} \Ub^+\Uio B_J.
       \end{gather}
       Moreover, $p(\underline{B})$ is independent of the choice of $\kappa_i$ for all $i\in \I_\ns$.
\end{lem}
\begin{proof}
    By Condition (a) under \cref{eq:Ins}, $\kappa_i$ can be nonzero only when $i\in \Ieven$. Then
    this follows from a reformulation of \cite[Prop. 4.12]{SW25}. Set $\Xi=p(\underline{B})$ and $Z=P_{-\lambda}(\Xi)$. It follows from \cite[(4.7)]{SW25} and \cref{eq:EjBk} that
   \begin{equation}
\label{comultiXi}
    \Delta(\Xi)\in\Xi\otimes K_{-\lambda}+\sum_{\{J\mid \wt(J)<\lambda\}}\Ualg^{+}_\bu\Ualg^{\imath0}B_J\otimes \Ualg.
\end{equation}
Hence \cref{eq:deltaP} implies that
   \begin{equation}
\label{comultiZ}
    \Delta(Z)\in \Xi\otimes K_{-\lambda}+\sum_{\{J\mid \wt(J)<\lambda\}}\Ualg^{+}_\bu\Ualg^{\imath0}B_J\otimes P_{-\lambda}(\Ualg).
\end{equation}
Since $Z=0$ by \cref{Proj=0}, after applying $1\otimes \epsilon$ to \cref{comultiZ} we conclude that \(\Xi\in \sum_{\{J\in\mathcal{J}|\wt(J)<\lambda\}} \Ub^+\Uio B_J.\) The independence from the parameters $\kappa_i$ follows from the fact that 
\[
\Delta(B_i)=B_i\otimes K_i^{-1}+ 1\otimes (F_i+\va_i E_iK_i^{-1}),\quad \forall i\in\I_\ns.
\]
Thus, under the stated assumptions, $\Delta(B_i)-B_i\otimes K_i^{-1}$ depends on $\va_i$ but not on $\kappa_i$. By the above proof, $p(\underline{B})$ is independent of the choice of $\kappa_i$ as well.
\end{proof}

Thus, obtaining an explicit Serre presentation of \(\Ui\) reduces to computing, for each quantum Serre polynomial \(p(\underline{B})\), the explicit element appearing on the right-hand side of \cref{remain}. For a quantum Serre polynomial
\(
p=p(x_{i_1},x_{i_2},\ldots,x_{i_\ell}),
\)
we denote by \(\Rem^p_{i_1,i_2,\ldots,i_\ell}\) the corresponding remainder term in \cref{remain}, that is,
\[
p(\underline{B})
=
\Rem^p_{i_1,i_2,\ldots,i_\ell}
\in
\sum_{\{J\mid \wt(J)<\lambda\}}
\Ualg^{+}_\bu \Ualg^{\imath 0} B_J .
\]
Here we suppress from the notation the dependence of \(\Rem^p_{i_1,i_2,\ldots,i_\ell}\) on the parameters.

Set $\Xi=p(\underline{B})$ and $Z=P_{-\lambda}(\Xi)$. By \cref{comultiZ} again, we have
\begin{equation}
    \label{key}
        \Rem^p_{i_1,i_2,\ldots,i_\ell}=-(\id\otimes \epsilon)(\Delta(Z)-\Xi\otimes K_{-\lambda})
        \overset{\cref{eq:deltaP}}{=}-(\id\otimes \epsilon)(\id \otimes (P_{-\lambda}\circ \pi_{0,0}))(\Delta(\Xi)-\Xi\otimes K_{-\lambda}).
\end{equation}
The expression
\[
(\id \otimes (P_{-\lambda}\circ \pi_{0,0}))(\Delta(\Xi)-\Xi\otimes K_{-\lambda})
\]
can often be evaluated rather efficiently, since many terms in
\(
(\id \otimes (P_{-\lambda}\circ \pi_{0,0}))\Delta(\Xi)
\)
vanish.  
More explicitly, a summand of \(\Delta(\Xi)\) can contribute nontrivially only if its second tensor factor contains the same number of \(E_i\)'s and \(F_i\)'s for some $i\in I$. Thus, in order to apply \cref{key} later, we need to have a better understanding of $\Delta(B_i)$ for $i\in I_\circ$.
\begin{lem}
(cf. \cite[Lem. 7.2]{Kol14})\label{lem:remain34}
    For any $i\in I_\circ$, we have
    \begin{gather}
        \label{comultiBi}
        \Delta(B_i)=B_i\otimes K_i^{-1}+1\otimes F_i + \va_i\mathcal{Z}_i \otimes E_{\tau i}K_i^{-1}+\mathcal{Q}
    \end{gather}
    for some $\mathcal{Z}_i\in \Ub\Uio$ and $\mathcal{Q}\in \Ub\Uio\otimes \sum_{\gamma>\alpha_{\tau i}}\Ualg_\gamma^+ K_i^{-1}$. 
\end{lem}
\begin{proof}
Recall that
\(
B_i=F_i+\va_i T_{w_\bu}(E_{\tau i})K_i^{-1}+\kappa_i K_i^{-1}.
\)
Kolb's proof of \cite[Lem.~7.2]{Kol14} relies on an alternative expression for the term \(T_{w_\bu}(E_{\tau i})K_i^{-1}\). By \cite[(3.16) and (4.2)]{SW25}, the analogous expression is available in the super setting, and hence the same argument applies here without essential changes. 
\end{proof}

\begin{lem}
(cf. \cite[Lem. 7.7]{Kol14})\label{lem:remain35}
    Suppose $i=\tau i\in I_\circ$ and $j\in I_\bu$. Then there exists $W_{i,j}\in \Ub^+$ such that
    \begin{gather}
        \label{comultiBi2}
        \Delta(B_i)=B_i\otimes K_i^{-1}+1\otimes F_i + \va_i\mathcal{Z}_i \otimes E_{\tau i}K_i^{-1}
        +\va_iW_{i,j}K_j\otimes T_j(E_i)K_i^{-1}+
        \mathcal{Q}'
    \end{gather}
    for some $\mathcal{Q}'\in \Ub\Uio\otimes\sum_{\gamma>\alpha_{i}\atop \gamma\neq \alpha_i+\alpha_j}\Ualg_\gamma^+ K_i^{-1}$. 
\end{lem}

Summarizing, the element $\mathcal{Z}_i$ in \cref{lem:remain34} is uniquely determined by the summand in $\Delta(T_{w_\bu}(E_{\tau i})K_i^{-1})$ whose second tensor factor is $E_{\tau i}K_i^{-1}$, while the element $W_{i,j}$ in \cref{lem:remain35} is uniquely determined by the summand in $\Delta(T_{w_\bu}(E_{ i})K_i^{-1})$ whose second tensor factor is $T_j(E_{i})K_i^{-1}$.

\subsection{Serre presentation}
\label{sec:Serre presentation}
The following theorem is a formal consequence of the results in \cite[Sect.~4]{SW25}. It holds for $\imath$quantum enveloping superalgebras of arbitrary basic types and extends \cite[Thm.~7.4]{Let02} and \cite[Thm.~7.1]{Kol14} to the super setting. For a given Dynkin diagram \(\I\), we define
\begin{gather}
\label{SS}
    \mathcal{S}(\I)
    =
    \{\text{all quantum Serre polynomials associated with }\I\}.
\end{gather}
\begin{theo}
\label{presentation}
    Let $\tUi$ be the superalgebra freely generated over $\Ub^+\Uio$ by elements $\tB_i$ and parities $|\tB_i|=|i|$ for all $i\in \I$ and let $\Phi$ denote the canonical algebra homomorphism $\Phi: \tUi\to \Ui$ given by 
    \[
    \tB_i \mapsto B_i,\qquad x \mapsto x,\qquad \text{ for all }i\in I,\ x\in \Ub^+\Uio.
    \]
    Write $\tB_J=\tB_{i_1}\tB_{i_2}\cdots \tB_{i_l}$ for any $J=(i_1,i_2,\ldots,i_l)\in \I^l$. Then there exist elements 
    \[
    \widetilde{\Rem}^p_{i_1,i_2,\ldots,i_k} \in \sum_{\{J\in \mathcal{J}\mid \wt(J)<\lambda(p)\}}\Ub^+\Uio \tB_J
    \]
    for all $p=p(x_{i_1},x_{i_2},\ldots,x_{i_k})\in \mathcal{S}(I)$ of weight $\lambda(p)$, such that the kernel $\ker(\Phi)$ is the ideal of $\tUi$ generated by the following elements:
    \begin{gather}
    \label{Serre1}
        K_h \tB_i -q^{-\langle h,\alpha_i\rangle}\tB_i K_h, \qquad \text{ for all }h\in Y^\io,i\in \I, \\
        \label{Serre2}
        [E_i,\tB_j]-\delta_{ij}\frac{K_i-K_i^{-1}}{q_i-q_i^{-1}},\qquad \text{ for all } i\in \I_\bu, j\in \I,\\
        \label{Serre3}
        p(\tB_{i_1},\tB_{i_2},\ldots,\tB_{i_k})-\widetilde{\Rem}^p_{i_1,i_2,\ldots,i_k} \text{ for all }p=p(x_{i_1},x_{i_2},\ldots,x_{i_k})\in \mathcal{S}(I).
    \end{gather}
    Moreover, the formal expression of $\widetilde{\Rem}^p_{i_1,i_2,\ldots,i_k}$ is independent of the choice of $\kappa_i$ for any $i\in \I_\ns$.
\end{theo}
\begin{proof}
It follows from \cref{remain} that, for every
\(
p=p(x_{i_1},x_{i_2},\ldots,x_{i_k})\in \mathcal{S}(I)
\)
of weight \(\lambda(p)\), there exists an element
\(
\widetilde{\Rem}^p_{i_1,i_2,\ldots,i_k}
\in
\sum_{\{J\in \mathcal{J}\mid \wt(J)<\lambda(p)\}}
\Ub^+\Uio \tB_J
\)
such that
\[
p(\tB_{i_1},\tB_{i_2},\ldots,\tB_{i_k})
-
\widetilde{\Rem}^p_{i_1,i_2,\ldots,i_k}
\in
\ker(\Phi).
\]
Let \(L\) be the ideal of \(\tUi\) generated by the relations
\cref{Serre1,Serre2,Serre3} for this choice of
\(\widetilde{\Rem}^p_{i_1,i_2,\ldots,i_k}\). By \cref{eq:EjBk,KEF}, we have
\(
L\subseteq \ker(\Phi).
\)
Hence \(\Phi\) induces a well-defined surjective homomorphism
\[
\tUi/L\longrightarrow \Ui.
\]
As in the proof of \cite[Theorem~4.14]{SW25}, one shows that \(\tUi/L\), as a left $\Ub^+\Uio$-module, is spanned by the elements \(\tB_J\) for \(J\in \mathcal{J}\). It then follows from \cref{basisthm} that the above surjection is also injective. Finally, the formal independence of the choice of \(\kappa_i\) for \(i\in \I_\ns\) has already been established in \cref{lem:remain}.
\end{proof}

In the remainder of this section, we shall determine 
\[
 {\Rem}^p_{i_1,i_2,\ldots,i_k} =\Phi( \widetilde{\Rem}^p_{i_1,i_2,\ldots,i_k} )
\]
for any super Satake diagram of type A (and hence $p$ runs over the ones in \cref{Serrepoly}) using direct computations and the projection technique \cref{key}.

\begin{rem}
    The discussions in Subsections 3.2 and 3.3 are valid for quantum supersymmetric pairs of {\bf arbitrary basic} types.
\end{rem}

\subsection{Determining ${\Rem}^{p_0}_{i}$ and ${\Rem}^{p_1}_{i,j}$}
Recall the quantum Serre polynomial \[p_0=p_0(x_i)=x_i^2\] from \cref{Serrepoly} for any $i\in \Iodd$. By \cref{def:superad}(4), if $\tau i=i$, then there exists an even simple root $j(i)\in\I_\bu$ such that $a_{ij}\neq 0$. 

\begin{prop}
    \label{prop:remp0}
    Suppose $i\in \Iodd$. The remainder  ${\Rem}^{p_0}_{i}=\Phi(\widetilde{\Rem}^{p_0}_{i})$ from \cref{presentation} is given by
\begin{gather}
    \label{remp0}
        {\Rem}^{p_0}_{i}=
        \begin{cases}
            -a_{i,j(i)} q_i^{a_{i,j(i)}}\va_i E_{j(i)},& \text{if }i=\tau i , \\
            0, & \text{otherwise}.
        \end{cases}
    \end{gather}
\end{prop}
\begin{proof}
    Since $|i|=\overline{1}$, we must have $i\in \I_\circ$ by \cref{def:superad}. When $i\neq \tau i$, summands of $\Delta(B_i^2)$ involving the term $\mathcal{Q}$ in \cref{comultiBi} will never survive under $P_{-2\alpha_i}\circ \pi_{0,0}$. Also, since $i\neq \tau i$, the second tensor factor of $\Delta(B_i^2)-B_i^2\otimes K_{-2\alpha_i}$ can never have the same number of $E_k$'s and $F_k$'s for any $k\in \I$. Hence by \cref{key} and \cref{comultiBi} we must have ${\Rem}^{p_0}_{i}$=0 when $i\neq \tau i$.

    Now suppose that $i=\tau i$, and let $j=j(i)$ be an even simple root in $\I_\bu$ such that $a_{ij}\neq 0$. By the classification of the type A super Satake diagrams under consideration in \eqref{eq:splitA} and \eqref{AIIIdiagram}, such a $j$ is unique. Then by definition we have \[
    B_i=F_i+\va_i T_j(E_i) K_i^{-1}.
    \]
    By \eqref{Tidef}, we have $T_j(E_i) = E_j E_i-q_i^{a_{ij}} E_i E_j$. By a direct calculation we have 
    \begin{gather*}
        B_i^2\overset{\cref{KEF}-\cref{square}}{=}\va_i\Big(
           E_j[E_i,F_i]-q_i^{a_{ij}}[E_i,F_i]E_j
        \Big)K_i^{-1} 
        \overset{\cref{KEF},\cref{skunk}}{=}\va_i\frac{1-q_i^{2a_{ij}}}{q_i-q_i^{-1}} E_j=[-a_{ij}]q_i^{a_{ij}}\va_i E_j.
    \end{gather*} 
    Since $a_{ij}=\pm 1$, we have $[-a_{ij}]=-a_{ij}$. This concludes the proof.
\end{proof}

Next we recall the quantum Serre polynomial \[p_1=p_1(x_i,x_j)=[x_i,x_j]\] from \cref{Serrepoly} for any $i,j\in \I$ such that $i\nsim j$. Hence we have $a_{ij}=0$.
\begin{prop}
    \label{prop:remp1}
    Suppose $i,j\in \I$ such that $i\nsim j$. The remainder ${\Rem}^{p_1}_{i,j}=\Phi(\widetilde{\Rem}^{p_1}_{i,j})$ from \cref{presentation} is given by
    \begin{gather}
    \label{remp1}
        {\Rem}^{p_1}_{i,j}=
        \begin{cases}
            \frac{\va_i \mathcal{Z}_i-\va_j \mathcal{Z}_j}{q_j-q_j^{-1}},& \text{if }i=\tau j \in \I_\circ, \\
            0, & \text{otherwise}.
        \end{cases}
    \end{gather}
\end{prop}
\begin{proof}
    Recall the element $\mathcal{Z}_i$ from \cref{comultiBi}. It is uniquely determined by the summand in $\Delta(T_{w_\bu}(E_{\tau i})K_i^{-1})$ whose second tensor factor is $E_{\tau i}K_i^{-1}$. By \cref{key}, we have
    \[
     {\Rem}^{p_1}_{i,j}=-(\id\otimes\epsilon)(\id\otimes (P_{-\alpha_i-\alpha_j}\circ \pi_{0,0})) (\Delta([B_i,B_j])-[B_i,B_j]\otimes K_{-\alpha_i-\alpha_j}).
    \]
    If $i\neq \tau j$, then the second tensor factor of $(\Delta([B_i,B_j])-[B_i,B_j]\otimes K_{-\alpha_i-\alpha_j})$ cannot have the same number of $E_k$'s and $F_k$'s for any $k\in \I$, and hence we must have ${\Rem}^{p_1}_{i,j}=0$. If $i=\tau j\in \I_\bu$, then we have $j\in \I_\bu$ as well. Hence $B_i=F_i$ and $B_j=F_j$ and it follows from \cref{square} that $[F_i,F_j]=0$ when $a_{ij}=0$.
    
    Now suppose $i= \tau j\in \I_\circ$. Then we have
    \begin{align*}
       &\id\otimes (P_{-\alpha_i-\alpha_j}\circ \pi_{0,0})( \Delta(B_iB_j)-B_iB_j\otimes K_{-\alpha_i-\alpha_j})\\=& \id\otimes (P_{-\alpha_i-\alpha_j}\circ \pi_{0,0}) (\va_j\mathcal{Z}_j\otimes F_iE_i K_j^{-1}) \\
       =& \id\otimes (P_{-\alpha_i-\alpha_j}\circ \pi_{0,0}) (\va_j\mathcal{Z}_j\otimes (-1)^{|i|+1}\frac{K_i-K_i^{-1}}{q_i-q_i^{-1}} K_j^{-1}) 
       =(-1)^{|i|}\frac{\va_j\mathcal{Z}_j\otimes K_{-\alpha_i-\alpha_j}}{q_i-q_i^{-1}}.
    \end{align*}
    In particular, by our assumptions we have (see for example \cite[(3.10)]{SW25})
\begin{gather}
    \label{didtaui}
    d_i=(-1)^{|i|}d_{\tau i},\ \forall i\in \I
\end{gather} in this case. Thus by \cref{key} we have
\[
 {\Rem}^{p_1}_{i,j}=-(-1)^{|i|}\frac{\va_j \mathcal{Z}_j}{q_i-q_i^{-1}}+\frac{\va_i\mathcal{Z}_i}{q_j-q_j^{-1}}\overset{\cref{didtaui}}{=}\frac{\va_i \mathcal{Z}_i-\va_j \mathcal{Z}_j}{q_j-q_j^{-1}},
\]
where the last equality follows from the fact that $q_i-q_i^{-1}=(-1)^{|i|}(q_j-q_j^{-1})$.
\end{proof}

\subsection{Determining $\Rem_{i,j}^{p_2}$}
Recall the quantum Serre polynomial
\[
p_2=p_2(x_i,x_j)=[x_i,[x_i,x_j]_{q_i^{-1}}]_{q_i}
\]
from \cref{Serrepoly} for all $i\in \Ieven$ and $j\in I$ such that $i\sim j$.
\begin{prop}
    \label{prop:remp2}
    Suppose $i\in \Ieven$ and $j\in I$ such that $i\sim j$. The remainder ${\Rem}^{p_2}_{i,j}=\Phi(\widetilde{\Rem}^{p_2}_{i,j})$ from \cref{presentation} is given by
    \begin{gather}
    \label{remp2}
        {\Rem}^{p_2}_{i,j}=
        \begin{cases}
            q_i\va_i B_j,& \text{if } i=\tau i \in \I_\circ, j\in \I_\circ \\
           \va_i \Big(\frac{q_i^2B_j\mathcal{Z}_i-\mathcal{Z}_iB_j}{q_i-q_i^{-1}}+\frac{q_i+q_i^{-1}}{q_j-q_j^{-1}}W_{i,j}K_j \Big),& \text{if } i=\tau i \in \I_\circ, j\in \I_\bu \\
           -[2]_{q_i}(q_i\va_{j}\mathcal{Z}_j+q_i^{-2}\va_i\mathcal{Z}_i)B_i, & \text{if } i=\tau j \in \I_\circ \\ 
            0, & \text{otherwise}.
        \end{cases}
    \end{gather}
\end{prop}
\begin{proof}
Since we are in type~A, we have \(a_{ij}=\pm 1\). If \(i\in \I_\bu\), then \cref{key} implies that
\(
\Rem^{p_2}_{i,j}=0.
\)
We may therefore assume that \(i\in \I_\circ\). If \(i\neq \tau i\), then \cref{key} implies that
\(
\Rem^{p_2}_{i,j}=0
\)
unless \(j=\tau i\).

It remains to consider the cases in which either \(i=\tau j\in \I_\circ\), or \(i=\tau i\in \I_\circ\). When $j\in I_\circ$, the formulas follow from a calculation analogous to that in \cite[Theorem~7.4, Case~2]{Kol14}.  
Finally, the remaining case, where \(i=\tau i\in \I_\circ\) and \(j\in \I_\bu\), follows from a calculation analogous to that in \cite[Theorem~7.8, Case~2]{Kol14}.
\end{proof}

For the type A super Satake diagrams in \cref{eq:splitA,AIIIdiagram}, the elements $\mathcal{Z}_i$ and $W_{i,j}$ admit the following explicit expressions. In fact, assuming that $i\in \Ieven\cap \I_\circ$ and $i\sim j$, we write down the possible subdiagrams and the corresponding $\mathcal{Z}_i$ and $W_{i,j}$ as follows.

\begin{itemize}
    \item $i=\tau i\in I_\circ$, $j\in I_\circ$: All possible local subdiagrams are 
    
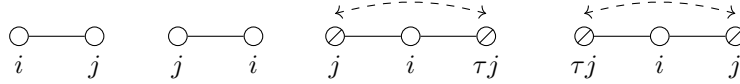
\[
    \begin{tikzpicture}
            \node (-2) [circ,
                label={below:{$i$}}] at (-2,0) {};
            \node (-1) [circ,
                label={below:{$j$}}] at (-1,0) {};
            \path (-2) edge (-1);
        \end{tikzpicture}
        \qquad
            \begin{tikzpicture}
            \node (-2) [circ,
                label={below:{$j$}}] at (-2,0) {};
            \node (-1) [circ,
                label={below:{$i$}}] at (-1,0) {};
            \path (-2) edge (-1);
        \end{tikzpicture}
        \qquad
        \begin{tikzpicture}
            \node (-2) [circslash,
                label={below:{$j$}}] at (-2,0) {};
            \node (-1) [circ,
                label={below:{$i$}}] at (-1,0) {};
             \node (0) [circslash,
                label={below:{$\tau j$}}] at (0,0) {};
            \path (-2) edge (-1)
                    (-1) edge (0);
             \draw[<->, dashed, bend left=15] (-2, .3) to (0, .3); 
        \end{tikzpicture}
           \qquad
        \begin{tikzpicture}
            \node (-2) [circslash,
                label={below:{$\tau j$}}] at (-2,0) {};
            \node (-1) [circ,
                label={below:{$i$}}] at (-1,0) {};
             \node (0) [circslash,
                label={below:{$j$}}] at (0,0) {};
            \path (-2) edge (-1)
                    (-1) edge (0);
             \draw[<->, dashed, bend left=15] (-2, .3) to (0, .3); 
        \end{tikzpicture}
    \]
    In the above four cases we always have $\mathcal{Z}_i=1$.
    \item $i=\tau i\in I_\circ,j\in \I_\bu$:
     The only possible local subdiagram is
     
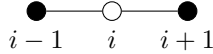
\[
      \begin{tikzpicture}
            \node (-2) [circblack,
               label={below:{$i-1$}}] at (-2,0) {};
            \node (-1) [circ,
                label={below:{$i$}}] at (-1,0) {};
             \node (0) [circblack,
                label={below:{$i+1$}}] at (0,0) {};
            \path (-2) edge (-1)
                    (-1) edge (0);
        \end{tikzpicture}
     \]
     with $j$ being one of the black nodes. In this case we have
     \(B_i=F_i+\va_i T_{i-1}T_{i+1}(E_i)K_i^{-1}\) and by calculating $\Delta(T_{i-1}T_{i+1}(E_i)K_i^{-1})$ directly one gets
     \[
     \mathcal{Z}_i=(1-q_i^{-2})^2E_{i-1}E_{i+1},\qquad W_{i,j}=(1-q_i^{-2})E_{2i-j}.
     \]
     These formulas agree with the purely even case considered in \cite[Ex. 7.9]{Kol14}.
     \item  $i=\tau j\in I_\circ$:
     The only possible local subdiagrams are
     \[
      \begin{tikzpicture}
            \node (-2) [circ,
               label={below:{$i$}}] at (-2,0) {};
            \node (-1) [circ,
                label={below:{$j$}}] at (-1,0) {};
            \path (-2) edge (-1);
             \draw[<->, dashed, bend left=15] (-2, .3) to (-1, .3); 
        \end{tikzpicture}
        \qquad
         \begin{tikzpicture}
            \node (-2) [circ,
               label={below:{$j$}}] at (-2,0) {};
            \node (-1) [circ,
                label={below:{$i$}}] at (-1,0) {};
            \path (-2) edge (-1);
             \draw[<->, dashed, bend left=15] (-2, .3) to (-1, .3); 
        \end{tikzpicture}
     \]
     In both cases we have $\mathcal{Z}_i=K_{\tau i}K_i^{-1}$ and $\mathcal{Z}_j=K_{\tau j}K_j^{-1}$.
\end{itemize}

\subsection{Determining $\Rem_{i,j,k}^{p_3}$}
So far, we have determined the elements in \cref{remp0,remp1,remp2}. The formulas obtained there are stated in a general form and apply to quantum supersymmetric pairs of arbitrary basic types, provided that their local Satake subdiagrams satisfy the above assumptions. In fact, the formulas in \cref{remp1,remp2} are similar to their purely even counterparts.   


Recall $p_3$ from \cref{Serrepoly}. In this subsection, we turn to the determination of the more complicated elements
\(
\Rem_{i,j,k}^{p_3},
\)
which are genuinely new to the super setting. Since it is cumbersome to describe the conditions on the three indices \(i,j,k\) directly, we present the nontrivial \( \Rem_{i,j,k}^{p_3} \) according to the possible subdiagrams appearing in \cref{eq:splitA,AIIIdiagram}. To simplify the notation in the diagrams, we use a dashed arrow attached to a node to indicate that the node is not fixed by $\tau$.  
For example, we may draw
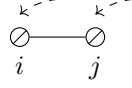
\[
          \begin{tikzpicture}
            \node (-2) [circslash,
               label={below:{$i$}}] at (-2,0) {};
            \node (-1) [circslash,
                label={below:{$j$}}] at (-1,0) {};
            \path (-2) edge (-1);
             \draw[<-, dashed, bend left=15] (-2, .3) to (-1.5, .5); 
             \draw[<-, dashed, bend left=15] (-1, .3) to (-.5, .5); 
        \end{tikzpicture}
\]
to mean that we have $i\neq \tau i$ and $j\neq \tau j$. For any $i\neq \tau i$, we define $\ck_i:=K_iK_{\tau i}^{-1}$ such that $\ck_{\tau i}=\ck_i^{-1}$.

\begin{prop}
\label{prop:remp3}
    Suppose $j\in \Iodd$, $i=j-1$ and $k=j+1$. Then the remainder ${\Rem}^{p_3}_{i,j,k}=\Phi(\widetilde{\Rem}^{p_3}_{i,j,k})$ from \cref{presentation} is given by
    \begin{gather}
    \label{remp3}
        {\Rem}^{p_3}_{i,j,k}=
        \begin{cases}
            0, & \{i,j,k\}\cap \{\tau i,\tau j,\tau k\}\in \{\varnothing,\{i\},\{k\}\}, \\
        (-1)^{|i|}\Big(q_{\tau j}^{-1}\va_{\tau j}[B_{j},B_{i}]_{q_j}\ck_{j}+\va_{j}[B_{j},B_{i}]_{q_{\tau j}}\ck_{\tau j}\Big)  &   \begin{tikzpicture}[anchorbase,scale=.7]
            \node (-2) [circslash,
                label={below:{$i$}}] at (-2,0) {};
            \node (-1) [circcross,
                label={below:{$j$}}] at (-1,0) {};
             \node (0) [circcross,
                label={below:{$k$}}] at (0,0) {};
            \path (-2) edge (-1)
                    (-1) edge (0);
             \draw[<->, dashed, bend left=15] (-1, .3) to (0, .3); 
             \draw[<-, dashed, bend left=15] (-2, .3) to (-1.5, .5);
        \end{tikzpicture} \\
        (-1)^{|k|}
\left(
q_j^{-1}\va_{\tau j}[B_j,B_k]_{q_{\tau j}}\ck_j
+
\va_j[B_j,B_k]_{q_j}\ck_{\tau j}
\right)
        &  \begin{tikzpicture}[anchorbase,scale=.7,xscale=-1]
            \node (-2) [circslash,
                label={below:{$k$}}] at (-2,0) {};
            \node (-1) [circcross,
                label={below:{$j$}}] at (-1,0) {};
             \node (0) [circcross,
                label={below:{$i$}}] at (0,0) {};
            \path (-2) edge (-1)
                    (-1) edge (0);
             \draw[<->, dashed, bend left=15] (-1, .3) to (0, .3); 
             \draw[<-, dashed, bend left=15] (-2, .3) to (-1.5, .5);
        \end{tikzpicture} \\
        q_j^{-1}\va_j B_k K_i^{-1}
        & 
        	\begin{tikzpicture}[anchorbase,scale=.7]
            \node (-1) [circ,fill=black,
                label={below:{$i$}}] at (-.5,0) {};
            \node (1) [circcross,
                label={below:{$j$}}] at (.5,0) {};
            \node (2) [circ,
                label={below:{$k$}}] at (1.5,0) {};
            \path
                (-1) edge (1)
                (1) edge (2);
        \end{tikzpicture} \\
        -q_j\va_j B_i K_k
        &
        \begin{tikzpicture}[anchorbase,scale=.7]
            \node (-1) [circ,
                label={below:{$i$}}] at (-.5,0) {};
            \node (1) [circcross,
                label={below:{$j$}}] at (.5,0) {};
            \node (2) [circ,fill=black,
                label={below:{$k$}}] at (1.5,0) {};
            \path
                (-1) edge (1)
                (1) edge (2);
        \end{tikzpicture} \\
     -q_j^{-1}\va_j B_kK_i^{-1}   & 
        	\begin{tikzpicture}[anchorbase,scale=.7]
            \node (-2) [circ,fill=black,
                label={below:{$i$}}] at (-1.5,0) {};
            \node (-1) [circcross,
                label={below:{$j$}}] at (-.5,0) {};
            \node (1) [circcross,
                label={below:{$k$}}] at (.5,0) {};
            \node (2) [circ,fill=black,
                label={below:{}}] at (1.5,0) {};
            \path
                (-2) edge (-1)
                (-1) edge (1)
                (1) edge (2);
        \end{tikzpicture} \\
     -q_j\va_j B_iK_k   &
        \begin{tikzpicture}[anchorbase,scale=.7]
            \node (-2) [circ,fill=black,
               ] at (-1.5,0) {};
            \node (-1) [circcross,
                label={below:{$i$}}] at (-.5,0) {};
            \node (1) [circcross,
                label={below:{$j$}}] at (.5,0) {};
            \node (2) [circ,fill=black,
                label={below:{$k$}}] at (1.5,0) {};
            \path
                (-2) edge (-1)
                (-1) edge (1)
                (1) edge (2);
        \end{tikzpicture} 
        \end{cases}
    \end{gather}
\end{prop}
\begin{proof}
  By \cref{key}, we have 
    \[
           {\Rem}^{p_3}_{i,j,k}=-(\id\otimes\epsilon)(\id\otimes (P_{-\alpha_i-2\alpha_j-\alpha_k}\circ \pi_{0,0})) (\Delta(p_3(B_i,B_j,B_k))-p_3(B_i,B_j,B_k)\otimes K_{-\alpha_i-2\alpha_j-\alpha_k})
    \]
     Now suppose that $\{i,j,k\}\cap \{\tau i,\tau j,\tau k\}\in \{\varnothing,\{i\},\{k\}\}$. In this case, by \cref{comultiBi} we must have ${\Rem}^{p_3}_{i,j,k}=0$. The remaining cases will be proved in the next subsection.
\end{proof}

\subsection{Proof of \cref{prop:remp3}}
In this subsection, we finish the proof of \cref{prop:remp3} by a case-by-case computation. We divide the remaining proof into three cases, depending on the local Satake diagrams appearing in \cref{remp3}.

\subsubsection{Case I}
We will use the notations from \cref{prop:remp3}. We first consider the local Satake diagram of the following form
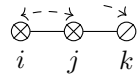
\[
\begin{tikzpicture}[anchorbase,scale=.7,xscale=-1]
            \node (-2) [circslash,
                label={below:{$k$}}] at (-2,0) {};
            \node (-1) [circcross,
                label={below:{$j$}}] at (-1,0) {};
             \node (0) [circcross,
                label={below:{$i$}}] at (0,0) {};
            \path (-2) edge (-1)
                    (-1) edge (0);
             \draw[<->, dashed, bend left=15] (-1, .3) to (0, .3); 
             \draw[<-, dashed, bend left=15] (-2, .3) to (-1.5, .5);
        \end{tikzpicture}
\]
In this case, we have $\I_\bu=\varnothing$ and $j=\tau i\in \Iodd$. Recall the Cartan matrix from \cref{Cmatrix}. We need to show that
\begin{equation}
    \label{case1}
    p_3(B_i,B_j,B_k)=[[B_{i},B_{j}]_{q_{j}}, B_{k}]_{q_k},B_{j}]= (-1)^{|k|}
\left(
q_j^{-1}\va_{\tau j}[B_j,B_k]_{q_{\tau j}}\ck_j
+
\va_j[B_j,B_k]_{q_j}\ck_{\tau j}
\right).
\end{equation}
We write $B_t=F_t+G_t$ for $t=i,j,k$, where $G_t:=\va_t E_{\tau t} K_t^{-1}$, and substitute them into the left-hand side of \cref{case1} to obtain a sum of sixteen terms. The two pure terms vanish by the quantum Serre relations \cref{Serre}, i.e.,
\begin{equation}
\label{trivialtermscase1}
    [[F_{i},F_{j}]_{q_{j}}, F_{k}]_{q_k},F_{j}]=0,
    \qquad
   [[G_{i},G_{j}]_{q_{j}}, G_{k}]_{q_k},G_{j}]=0.
\end{equation}
We now record the remaining terms.  
\begin{lem}
\label{lem:case1}
    The following identities hold:
    \begin{gather*}
        \bigl[\bigl[\,[G_i,F_j]_{q_j},F_k\bigr]_{q_k},F_j\bigr]=(-1)^{|k|}q_j^{-1}\va_i [F_j,F_k]_{q_k}\ck_j, \\
         \bigl[\bigl[\,[F_i,G_j]_{q_j},F_k\bigr]_{q_k},F_j\bigr]=
         (-1)^{|k|}\va_j
    \left(
        [F_j,F_k]_{q_j}\ck_{\tau j}
        -
        [F_j,F_k]_{q_k}K_{\tau j}^{-2}\ck_{\tau j}
    \right),\\
     \bigl[\bigl[\,[F_i,F_j]_{q_j},G_k\bigr]_{q_k},F_j\bigr]=0
    \\
     \bigl[\bigl[\,[F_i,F_j]_{q_j},F_k\bigr]_{q_k},G_j\bigr]= (-1)^{|k|}\va_j [F_j,F_k]_{q_k}K_{\tau j}^{-2}\ck_{\tau j},  \\
    \bigl[\bigl[\,[G_i,G_j]_{q_j},F_k\bigr]_{q_k},F_j\bigr]
    =
    \va_i\va_j(q_j-q_j^{-1})
    \Bigl(
        -(-1)^{|k|}
        [F_j,F_k]_{q_k}
        \bigl(E_jE_i+q_jE_iE_j\bigr)K_i^{-1}K_j^{-1} 
        +
        q_j^{-1}F_kE_iK_i^{-1}
    \Bigr),\\
     \bigl[\bigl[\,[G_i,F_j]_{q_j},G_k\bigr]_{q_k},F_j\bigr]=0, \\
         \bigl[\bigl[\,[G_i,F_j]_{q_j},F_k\bigr]_{q_k},G_j\bigr]
    =
    (-1)^{|k|}\va_i\va_j(q_j-q_j^{-1})
    [F_j,F_k]_{q_k}
    \bigl(E_jE_i+q_jE_iE_j\bigr)K_i^{-1}K_j^{-1},\\
     \bigl[\bigl[\,[F_i,G_j]_{q_j},G_k\bigr]_{q_k},F_j\bigr]
    =
    -(-1)^{|k|}\va_j\va_k(q_j^2-1)
    F_jE_{\tau k}K_k^{-1}\ck_{\tau j},\\
     \bigl[\bigl[\,[F_i,G_j]_{q_j},F_k\bigr]_{q_k},G_j\bigr]=0= \bigl[\bigl[\,[F_i,F_j]_{q_j},G_k\bigr]_{q_k},G_j\bigr],\\
     \bigl[\bigl[\,[G_i,G_j]_{q_j},G_k\bigr]_{q_k},F_j\bigr]
    =
    (-1)^{|k|}q_k\va_i\va_j\va_k
    \bigl(E_iE_{\tau k}-(-1)^{|k|}q_kE_{\tau k}E_i\bigr)
    K_i^{-1}K_k^{-1},\\
     \bigl[\bigl[\,[G_i,G_j]_{q_j},F_k\bigr]_{q_k},G_j\bigr]=0= \bigl[\bigl[\,[G_i,F_j]_{q_j},G_k\bigr]_{q_k},G_j\bigr],\\
      \bigl[\bigl[\,[F_i,G_j]_{q_j},G_k\bigr]_{q_k},G_j\bigr]=
    (-1)^{|k|}\va_j^2\va_k
    \bigl(
        E_iE_{\tau k}-(-1)^{|k|}q_jE_{\tau k}E_i
    \bigr)
    K_iK_j^{-2}K_k^{-1}.
    \end{gather*}
\end{lem}
\begin{proof}
We shall use \cref{skunk} repeatedly, and in this case, we have 
\(
    q_i=q_k=q_j^{-1}
\)
and
\(
    F_j^2=0=E_i^2=E_j^2
\).
Set
\[
    H_r:=\frac{K_r-K_r^{-1}}{q_r-q_r^{-1}},
    \qquad
    C:=E_jE_i+q_jE_iE_j .
\]
A direct computation gives
\[
    [G_i,F_j]_{q_j}
=
    \va_i\bigl(q_j^{-1}H_j+(q_j-q_j^{-1})F_jE_j\bigr)K_i^{-1}, \quad
    [F_i,G_j]_{q_j}
    =
    \va_jH_iK_j^{-1}, \quad
    [G_i,G_j]_{q_j}
    =
    \va_i\va_jC K_i^{-1}K_j^{-1}.
\]
Taking the \(q_k\)-commutator with \(F_k\) or \(G_k\), we obtain
\begin{equation}
\label{ricecake}
\begin{aligned}
    \bigl[\,[G_i,F_j]_{q_j},F_k\bigr]_{q_k}
    &=
    \va_i
    \left(
        -q_j^{-1}F_kK_j^{-1}
        +(q_j-q_j^{-1})(-1)^{|k|}
        [F_j,F_k]_{q_k}E_j
    \right)K_i^{-1},\\
    \bigl[\,[F_i,G_j]_{q_j},F_k\bigr]_{q_k}
    &=
    \va_j(q_j-q_j^{-1})F_kH_iK_j^{-1},\\
    \bigl[\,[G_i,G_j]_{q_j},F_k\bigr]_{q_k}
    &=
    \va_i\va_j(q_j-q_j^{-1})
    F_kC K_i^{-1}K_j^{-1},\\
    \bigl[\,[F_i,G_j]_{q_j},G_k\bigr]_{q_k}
    &=
    -\va_j\va_kE_{\tau k}K_iK_j^{-1}K_k^{-1},\\
    \bigl[\,[G_i,G_j]_{q_j},G_k\bigr]_{q_k}
    &=
    q_k\va_i\va_j\va_k
    \bigl(CE_{\tau k}-q_kE_{\tau k}C\bigr)
    K_i^{-1}K_j^{-1}K_k^{-1}.
\end{aligned}
\end{equation}
Moreover,
\begin{equation}
\label{ricecake2}
        \bigl[\,[F_i,F_j]_{q_j},G_k\bigr]_{q_k}=0,
    \qquad
    \bigl[\,[G_i,F_j]_{q_j},G_k\bigr]_{q_k}=0.
\end{equation}
These two terms vanish because \(E_{\tau k}\) super-commutes with the relevant
\(F\)-terms and \(E\)-terms, and the \(K_k^{-1}\)-factor contributes precisely the
scalar killed by the external \(q_k\)-commutator.

It remains to take the supercommutator with the final entry. From \cref{ricecake}, one gets
\[
\begin{aligned}
    \bigl[\bigl[\,[G_i,F_j]_{q_j},F_k\bigr]_{q_k},F_j\bigr]
    &=
    (-1)^{|k|}q_j^{-1}\va_i
    [F_j,F_k]_{q_k}\ck_j,\\
    \bigl[\bigl[\,[G_i,F_j]_{q_j},F_k\bigr]_{q_k},G_j\bigr]
    &=
    (-1)^{|k|}\va_i\va_j(q_j-q_j^{-1})
    [F_j,F_k]_{q_k}C K_i^{-1}K_j^{-1}.
\end{aligned}
\]
Indeed, the first identity follows from
\(
    E_jF_j+F_jE_j=H_j\) and \( F_j^2=0\),
while the second one follows from
\[
    [F_kK_j^{-1}K_i^{-1},E_iK_j^{-1}]=0,
    \qquad
    [\, [F_j,F_k]_{q_k}E_jK_i^{-1},E_iK_j^{-1}]
    =
    [F_j,F_k]_{q_k}C K_i^{-1}K_j^{-1}.
\]

Similarly, we get from \cref{ricecake,ricecake2} that
\[
\begin{aligned}
    \bigl[\bigl[\,[F_i,G_j]_{q_j},F_k\bigr]_{q_k},F_j\bigr]
    &=
    (-1)^{|k|}\va_j
    \left(
        [F_j,F_k]_{q_j}\ck_{\tau j}
        -
        [F_j,F_k]_{q_k}K_{\tau j}^{-2}\ck_{\tau j}
    \right),\\
    \bigl[\bigl[\,[F_i,G_j]_{q_j},F_k\bigr]_{q_k},G_j\bigr]
    &=0.
\end{aligned}
\]
Here the second equation follows from the fact that \(H_i\)
commutes with \(E_i\).

Next,
\[
    \bigl[\bigl[\,[F_i,F_j]_{q_j},F_k\bigr]_{q_k},G_j\bigr]
    =
    (-1)^{|k|}\va_j
    [F_j,F_k]_{q_k}K_{\tau j}^{-2}\ck_{\tau j}.
\]
This follows by commuting \(E_i\) through
\([\, [F_i,F_j]_{q_j},F_k]_{q_k}\), using
\(
    [[F_i,F_j]_{q_j},E_i]=F_jK_i^{-1}.
\)

For the terms involving \([G_i,G_j]_{q_j}\), we need to use
\(
    CF_j=F_jC+E_iK_j.
\)
This gives
\[
    \bigl[\bigl[\,[G_i,G_j]_{q_j},F_k\bigr]_{q_k},F_j\bigr]
    =
    \va_i\va_j(q_j-q_j^{-1})
    \Bigl(
        -(-1)^{|k|}
        [F_j,F_k]_{q_k}C K_i^{-1}K_j^{-1}  
        +q_j^{-1}F_kE_iK_i^{-1}
    \Bigr),
\]
whereas
\[
    \bigl[\bigl[\,[G_i,G_j]_{q_j},F_k\bigr]_{q_k},G_j\bigr]=0,
\]
since
\(
    CE_i=q_jE_iC.
\)

It remains to compute the terms whose third entry is \(G_k\). By \cref{ricecake}, we have
\begin{align*}
    \bigl[\bigl[\,[F_i,G_j]_{q_j},G_k\bigr]_{q_k},G_j\bigr]  
    =&
    -\va_j^2\va_k
    \bigl[
        E_{\tau k}K_iK_j^{-1}K_k^{-1},
        E_iK_j^{-1}
    \bigr]                                      \\
 =&
    (-1)^{|k|}\va_j^2\va_k
    \bigl(
        E_iE_{\tau k}-(-1)^{|k|}q_jE_{\tau k}E_i
    \bigr)
    K_iK_j^{-2}K_k^{-1}
\end{align*}
and
\[
\begin{aligned}
\bigl[\bigl[\,[G_i,G_j]_{q_j},G_k\bigr]_{q_k},F_j\bigr]  
 =&
q_k\va_i\va_j\va_k
\bigl[
    CE_{\tau k}-q_kE_{\tau k}C,
    F_j
\bigr]
K_i^{-1}K_j^{-1}K_k^{-1}                         \\
=&
(-1)^{|k|}q_k\va_i\va_j\va_k
\bigl(
    E_iE_{\tau k}-(-1)^{|k|}q_kE_{\tau k}E_i
\bigr)
K_i^{-1}K_k^{-1}.
\end{aligned}
\]
Finally, using \cref{ricecake2} we get
\[
    \bigl[\bigl[\,[G_i,F_j]_{q_j},G_k\bigr]_{q_k},F_j\bigr]
    =
    \bigl[\bigl[\,[G_i,F_j]_{q_j},G_k\bigr]_{q_k},G_j\bigr]
    =
    0.
\]
This concludes the proof.
\end{proof}

Summing the two terms in \cref{trivialtermscase1} and the fourteen terms in
\cref{lem:case1}, the terms involving
\[
    [F_j,F_k]_{q_k}
    \bigl(E_jE_i+q_jE_iE_j\bigr)K_i^{-1}K_j^{-1}
\]
cancel. The two terms involving
\(
    [F_j,F_k]_{q_k}K_{\tau j}^{-2}\ck_{\tau j}
\)
also cancel. Hence we get
\[
\begin{aligned}
p_3(B_i,B_j,B_k)
&=
(-1)^{|k|}q_j^{-1}\va_i [F_j,F_k]_{q_k}\ck_j
+
(-1)^{|k|}\va_j [F_j,F_k]_{q_j}\ck_{\tau j} \\
&\quad
+\va_i\va_j(q_j-q_j^{-1})q_j^{-1}F_kE_iK_i^{-1} \\
&\quad
-(-1)^{|k|}\va_j\va_k(q_j^2-1)
F_jE_{\tau k}K_k^{-1}\ck_{\tau j} \\
&\quad
+(-1)^{|k|}q_k\va_i\va_j\va_k
\bigl(E_iE_{\tau k}-(-1)^{|k|}q_kE_{\tau k}E_i\bigr)
K_i^{-1}K_k^{-1} \\
&\quad
+(-1)^{|k|}\va_j^2\va_k
\bigl(E_iE_{\tau k}-(-1)^{|k|}q_jE_{\tau k}E_i\bigr)
K_iK_j^{-2}K_k^{-1}.
\end{aligned}
\]
We now rewrite the remaining terms. Recall that
\[
    G_j=\va_jE_iK_j^{-1},\qquad
    G_k=\va_kE_{\tau k}K_k^{-1},\qquad
    \ck_j=K_jK_i^{-1},\qquad
    \ck_{\tau j}=K_iK_j^{-1}.
\]
A direct calculation gives
\[
\begin{aligned}
    [G_j,F_k]_{q_k}
    &=
    (-1)^{|k|}\va_j(q_j-q_j^{-1})F_kE_iK_j^{-1},\\
    [F_j,G_k]_{q_j}
    &=
    -\va_k(q_j^2-1)F_jE_{\tau k}K_k^{-1},\\
    [G_j,G_k]_{q_k}
    &=
    \va_j\va_k
    \bigl(E_iE_{\tau k}-(-1)^{|k|}q_kE_{\tau k}E_i\bigr)
    K_j^{-1}K_k^{-1},\\
    [G_j,G_k]_{q_j}
    &=
    \va_j\va_k
    \bigl(E_iE_{\tau k}-(-1)^{|k|}q_jE_{\tau k}E_i\bigr)
    K_j^{-1}K_k^{-1}.
\end{aligned}
\]
Therefore the preceding displayed expression becomes
\[
\begin{aligned}
p_3(B_i,B_j,B_k)
&=
(-1)^{|k|}q_j^{-1}\va_i
\bigl(
    [F_j,F_k]_{q_k}
    +
    [G_j,F_k]_{q_k}
    +
    [G_j,G_k]_{q_k}
\bigr)\ck_j \\
&\quad
+
(-1)^{|k|}\va_j
\bigl(
    [F_j,F_k]_{q_j}
    +
    [F_j,G_k]_{q_j}
    +
    [G_j,G_k]_{q_j}
\bigr)\ck_{\tau j}.
\end{aligned}
\]
Finally, we have
\[
    [F_j,G_k]_{q_k}=0,
    \qquad
    [G_j,F_k]_{q_j}=0.
\]
Indeed, the first equality follows from \(q_kq_j=1\), and the second one follows
from the same relation after moving \(K_j^{-1}\) past \(F_k\). Hence
\[
    [F_j,F_k]_{q_k}+[G_j,F_k]_{q_k}+[G_j,G_k]_{q_k}
    =
    [B_j,B_k]_{q_k},\quad
    [F_j,F_k]_{q_j}+[F_j,G_k]_{q_j}+[G_j,G_k]_{q_j}
    =
    [B_j,B_k]_{q_j}.
\]
Consequently,
\[
p_3(B_i,B_j,B_k)
=
(-1)^{|k|}
\left(
    q_j^{-1}\va_i[B_j,B_k]_{q_k}\ck_j
    +
    \va_j[B_j,B_k]_{q_j}\ck_{\tau j}
\right).
\]
This proves \cref{case1}. The calculation for the local Satake diagram of the form
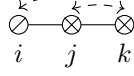
\[
 \begin{tikzpicture}[anchorbase,scale=.7]
            \node (-2) [circslash,
                label={below:{$i$}}] at (-2,0) {};
            \node (-1) [circcross,
                label={below:{$j$}}] at (-1,0) {};
             \node (0) [circcross,
                label={below:{$k$}}] at (0,0) {};
            \path (-2) edge (-1)
                    (-1) edge (0);
             \draw[<->, dashed, bend left=15] (-1, .3) to (0, .3); 
             \draw[<-, dashed, bend left=15] (-2, .3) to (-1.5, .5);
        \end{tikzpicture}
\]
is similar and omitted.

\subsubsection{Case II}
Consider the following local Satake diagram:
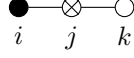
\[
        	\begin{tikzpicture}[anchorbase,scale=.7]
            \node (-1) [circ,fill=black,
                label={below:{$i$}}] at (-.5,0) {};
            \node (1) [circcross,
                label={below:{$j$}}] at (.5,0) {};
            \node (2) [circ,
                label={below:{$k$}}] at (1.5,0) {};
            \path
                (-1) edge (1)
                (1) edge (2);
        \end{tikzpicture}
\]
In this case, we have $\tau=\id$ and $i\in I_\bu$. Hence by \cref{def:Ui} we have
\[
B_i=F_i,\qquad
B_j=F_j+\va_j T_i(E_j)K_j^{-1},\qquad B_k=F_k+\va_k E_kK_k^{-1}.
\]
We need to show that
\begin{equation}
    \label{case2}
    p_3(F_i,B_j,B_k)=[[F_{i},B_{j}]_{q_{j}}, B_{k}]_{q_k},B_{j}]= q_j^{-1}\va_j B_k K_i^{-1}.
\end{equation}
Since $a_{ij}=-(-1)^{|i|}=-1$ by \cref{Cmatrix}, we define 
\[
G_j:=\va_jT_i(E_j)K_j^{-1}=\va_j(E_iE_j-q_i^{a_{ij}}E_jE_i)K_j^{-1}=\va_j[E_i,E_j]_{q_i^{-1}}K_j^{-1},\qquad G_k=\va_kE_kK_k^{-1}.
\]
After writing $B_r=F_r+G_r$ for $r=j,k$ and substituting them into the left-hand side of \cref{case2}, we obtain a sum of eight terms. We record them as follows.
\begin{lem}
    \label{lem:case2}
    We calculate that
    \begin{gather*}
            \bigl[\bigl[\,[F_i,F_j]_{q_j},F_k\bigr]_{q_k},F_j\bigr]=0,\qquad
                \bigl[\bigl[\,[F_i,F_j]_{q_j},G_k\bigr]_{q_k},F_j\bigr]=0, \\
                    \bigl[\bigl[\,[F_i,G_j]_{q_j},F_k\bigr]_{q_k},F_j\bigr]=\va_j(q_j-q_j^{-1})[F_j,F_k]_{q_k}E_jK_i^{-1}K_j^{-1}
                    +
                        q_j^{-1}\va_jF_kK_i^{-1}(1-K_j^{-2}), \\
             \bigl[\bigl[\,[F_i,F_j]_{q_j},F_k\bigr]_{q_k},G_j\bigr]=-\va_j
            \Bigl(q_j[F_j,F_k]_{q_k}E_j + q_j^{-1}E_j[F_j,F_k]_{q_k}\Bigr)K_i^{-1}K_j^{-1},\\
            \bigl[\bigl[\,[F_i,G_j]_{q_j},G_k\bigr]_{q_k},F_j\bigr]
                    =
                    q_j^{-1}\va_j\va_k E_kK_i^{-1}K_k^{-1},\\
                     \bigl[\bigl[\,[F_i,G_j]_{q_j},F_k\bigr]_{q_k},G_j\bigr]=0,\quad
                      \bigl[\bigl[\,[F_i,F_j]_{q_j},G_k\bigr]_{q_k},G_j\bigr]=0,\quad
                       \bigl[\bigl[\,[F_i,G_j]_{q_j},G_k\bigr]_{q_k},G_j\bigr]=0.
    \end{gather*}
\end{lem}
    \begin{proof}
    By the quantum Serre relation \cref{Serre}, we have
\[
    \bigl[\bigl[\,[F_i,F_j]_{q_j},F_k\bigr]_{q_k},F_j\bigr]=0.
\]
Since \(i\in I_\bullet\) is even and \(i\sim j\), we have
\[
\begin{aligned}
    [F_i,G_j]_{q_j}
    &=
    \va_jE_jK_i^{-1}K_j^{-1}.
\end{aligned}
\]
Hence
\begin{equation}
\label{shrimp1}
\begin{aligned}
    \bigl[\,[F_i,G_j]_{q_j},F_k\bigr]_{q_k}
    &=
    \va_j(q_j-q_j^{-1})F_kE_jK_i^{-1}K_j^{-1},\\
    \bigl[\,[F_i,G_j]_{q_j},G_k\bigr]_{q_k}
    &=
    q_j^{-1}\va_j\va_k[E_j,E_k]_{q_k}
    K_i^{-1}K_j^{-1}K_k^{-1}.
\end{aligned}
\end{equation}
Also, since \(E_k\) super-commutes with both \(F_i\) and \(F_j\), and the
\(K_k^{-1}\)-factor contributes exactly the scalar killed by the outer
\(q_k\)-commutator, we have
\begin{equation}
    \label{shrimp2}
    \bigl[\,[F_i,F_j]_{q_j},G_k\bigr]_{q_k}=0.
\end{equation}
Therefore by \cref{shrimp2} we conclude that
\[
    \bigl[\bigl[\,[F_i,F_j]_{q_j},G_k\bigr]_{q_k},F_j\bigr]=0
    =
    \bigl[\bigl[\,[F_i,F_j]_{q_j},G_k\bigr]_{q_k},G_j\bigr].
\]

We now compute the remaining terms. First, by \cref{shrimp1} we have
\[
\bigl[\bigl[\,[F_i,G_j]_{q_j},F_k\bigr]_{q_k},F_j\bigr]
=
\va_j(q_j-q_j^{-1})
\left(
    q_j^{-1}F_kE_jF_j+F_jF_kE_j
\right)K_i^{-1}K_j^{-1}.
\]
Using
\(
    E_jF_j=-F_jE_j+\frac{K_j-K_j^{-1}}{q_j-q_j^{-1}},
\)
this becomes
\[
\bigl[\bigl[\,[F_i,G_j]_{q_j},F_k\bigr]_{q_k},F_j\bigr]\\ =
\va_j(q_j-q_j^{-1})[F_j,F_k]_{q_k}E_jK_i^{-1}K_j^{-1}
+
q_j^{-1}\va_jF_kK_i^{-1}(1-K_j^{-2}).
\]
Next we show that \(
    \bigl[\bigl[\,[F_i,G_j]_{q_j},F_k\bigr]_{q_k},G_j\bigr]=0
\). 
Put
\[
    C:=[E_i,E_j]_{q_j^{-1}}.
\]
By \cref{shrimp1} we have
\[
\bigl[\bigl[\,[F_i,G_j]_{q_j},F_k\bigr]_{q_k},G_j\bigr]\\ =
\va_j^2(q_j-q_j^{-1})
\bigl[
    F_kE_jK_i^{-1}K_j^{-1},
    CK_j^{-1}
\bigr].
\]
Moving the \(K\)-factors gives
\[
\bigl[
    F_kE_jK_i^{-1}K_j^{-1},
    CK_j^{-1}
\bigr]
=
F_k\bigl(E_jC+q_jCE_j\bigr)K_i^{-1}K_j^{-2}.
\]
But a direct computation, using $E_j^2=0$, shows that
\[
    E_jC+q_jCE_j=0.
\]
Hence we conclude that
\[
    \bigl[\bigl[\,[F_i,G_j]_{q_j},F_k\bigr]_{q_k},G_j\bigr]=0.
\]

For the remaining terms, note that \(K_j^{-1}\) commutes with
\(\bigl[\,[F_i,F_j]_{q_j},F_k\bigr]_{q_k}\), since
\(
    a_{ji}+a_{jj}+a_{jk}=0.
\)
Moreover, we see that
\[
\bigl[
    [F_i,F_j]_{q_j},
    C
\bigr]
 =
-\bigl(q_jF_jE_j+q_j^{-1}E_jF_j\bigr)K_i^{-1}+q_j^{-1}F_iE_iK_j-E_iK_jF_i.
\]
Hence
\[
\begin{aligned}
\bigl[
    \bigl[\,[F_i,F_j]_{q_j},F_k\bigr]_{q_k},
    C
\bigr]
 =&
-\bigl(q_jF_jE_j+q_j^{-1}E_jF_j\bigr)K_i^{-1}F_k
+
q_kF_k\bigl(q_jF_jE_j+q_j^{-1}E_jF_j\bigr)K_i^{-1}\\
 =&
-
\Bigl(
    q_j[F_j,F_k]_{q_k}E_j
    +
    q_j^{-1}E_j[F_j,F_k]_{q_k}
\Bigr)K_i^{-1}.
\end{aligned}
\]
Therefore
\[
\begin{aligned}
    \bigl[\bigl[\,[F_i,F_j]_{q_j},F_k\bigr]_{q_k},G_j\bigr]
    &=
    -\va_j
    \Bigl(
        q_j[F_j,F_k]_{q_k}E_j
        +
        q_j^{-1}E_j[F_j,F_k]_{q_k}
    \Bigr)K_i^{-1}K_j^{-1}.
\end{aligned}
\]

It remains to compute the terms involving \(G_k\). From \cref{shrimp1}
and
\(
    [[E_j,E_k]_{q_k},F_j]=E_kK_j,
\)
we obtain
\[
\begin{aligned}
    \bigl[\bigl[\,[F_i,G_j]_{q_j},G_k\bigr]_{q_k},F_j\bigr]
    &=
    q_j^{-1}\va_j\va_k E_kK_i^{-1}K_k^{-1}.
\end{aligned}
\]
Finally, the equality
\[
    \bigl[\bigl[\,[F_i,G_j]_{q_j},G_k\bigr]_{q_k},G_j\bigr]=0
\]
follows from \cite[Lem.~4.5(3)]{SW25}. This proves the lemma.
\end{proof}

Summing the eight terms in \cref{lem:case2}, we get
\begin{equation}
\label{taro2026}
\begin{aligned}
p_3(F_i,B_j,B_k)
&=
\va_j(q_j-q_j^{-1})[F_j,F_k]_{q_k}E_jK_i^{-1}K_j^{-1}
+
q_j^{-1}\va_jF_kK_i^{-1}(1-K_j^{-2}) \\
&\quad
-\va_j
\Bigl(
    q_j[F_j,F_k]_{q_k}E_j
    +
    q_j^{-1}E_j[F_j,F_k]_{q_k}
\Bigr)K_i^{-1}K_j^{-1} \\
&\quad
+
q_j^{-1}\va_j\va_kE_kK_i^{-1}K_k^{-1}.
\end{aligned}
\end{equation}
The first three terms on the right-hand side of \cref{taro2026} combine to
\begin{equation}
\label{taro}
    \begin{aligned}
&-\va_jq_j^{-1}
\Bigl(
    [F_j,F_k]_{q_k}E_j
    +
    E_j[F_j,F_k]_{q_k}
\Bigr)K_i^{-1}K_j^{-1}
+
q_j^{-1}\va_jF_kK_i^{-1}(1-K_j^{-2}).
\end{aligned}
\end{equation}
We calculate that
\[
\begin{aligned}
[F_j,F_k]_{q_k}E_j+E_j[F_j,F_k]_{q_k} 
 =&
(F_jF_k-q_j^{-1}F_kF_j)E_j
+
E_j(F_jF_k-q_j^{-1}F_kF_j) \\
 =&
(E_jF_j+F_jE_j)F_k
-
q_j^{-1}F_k(E_jF_j+F_jE_j)
 \\
=&
\frac{K_j-K_j^{-1}}{q_j-q_j^{-1}}F_k
-
q_j^{-1}F_k\frac{K_j-K_j^{-1}}{q_j-q_j^{-1}}
 =
-F_kK_j^{-1}.
\end{aligned}
\]
Substituting the above equation into \cref{taro}, we get
\[
-\va_jq_j^{-1}
\Bigl(
    [F_j,F_k]_{q_k}E_j
    +
    E_j[F_j,F_k]_{q_k}
\Bigr)K_i^{-1}K_j^{-1}
+
q_j^{-1}\va_jF_kK_i^{-1}(1-K_j^{-2})=
q_j^{-1}\va_jF_kK_i^{-1}.
\]
Hence
\[
\begin{aligned}
p_3(F_i,B_j,B_k)
=
q_j^{-1}\va_jF_kK_i^{-1}
+
q_j^{-1}\va_j\va_kE_kK_i^{-1}K_k^{-1}=
q_j^{-1}\va_jB_kK_i^{-1}.
\end{aligned}
\]
This completes the proof of \cref{case2}. The calculation for the local Satake diagram of the form
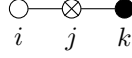
\[
\begin{tikzpicture}[anchorbase,scale=.7]
            \node (-1) [circ,
                label={below:{$i$}}] at (-.5,0) {};
            \node (1) [circcross,
                label={below:{$j$}}] at (.5,0) {};
            \node (2) [circ,fill=black,
                label={below:{$k$}}] at (1.5,0) {};
            \path
                (-1) edge (1)
                (1) edge (2);
        \end{tikzpicture}
\]
is similar and omitted.

\subsubsection{Case III}
Finally, we consider local Satake diagrams of the following form
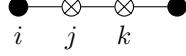
\[
\begin{tikzpicture}[anchorbase,scale=.7]
            \node (-2) [circ,fill=black,
                label={below:{$i$}}] at (-1.5,0) {};
            \node (-1) [circcross,
                label={below:{$j$}}] at (-.5,0) {};
            \node (1) [circcross,
                label={below:{$k$}}] at (.5,0) {};
            \node (2) [circ,fill=black,
                label={below:{}}] at (1.5,0) {};
            \path
                (-2) edge (-1)
                (-1) edge (1)
                (1) edge (2);
        \end{tikzpicture}
\]
We also label the rightmost black node by $l=k+1$ for convenience. In this case, we have $\tau=\id$ and $i,l\in I_\bu$. Hence by \cref{def:Ui} we have
\[
B_i=F_i,\qquad
B_j=F_j+\va_j T_i(E_j)K_j^{-1},\qquad B_k=F_k+\va_k T_l(E_k)K_k^{-1}.
\]
We need to show that
\begin{equation}
    \label{case3}
    p_3(B_i,B_j,B_k)=\bigl[\bigl[\,[B_i,B_j]_{q_j},B_k\bigr]_{q_k},B_j\bigr]
                =
            -q_j^{-1}\va_j B_kK_i^{-1}.
\end{equation}
Since $a_{ij}=-(-1)^{|i|}=-1$ and $a_{lk}=-1$ by \cref{Cmatrix}, we define 
\[
G_j:=\va_jT_i(E_j)K_j^{-1}=\va_j[E_i,E_j]_{q_i^{-1}}K_j^{-1},\qquad G_k=\va_kT_l(E_k)K_k^{-1}=\va_k [E_l,E_k]_{q_l^{-1}}K_k^{-1}.
\]
\begin{lem}
    \label{lem:case3}
    We calculate that
    \begin{gather*}
         \bigl[\bigl[\,[F_i,F_j]_{q_j},F_k\bigr]_{q_k},F_j\bigr]=0,\qquad
                \bigl[\bigl[\,[F_i,F_j]_{q_j},G_k\bigr]_{q_k},F_j\bigr]=0,
                \\
                 \bigl[\bigl[\,[F_i,G_j]_{q_j},F_k\bigr]_{q_k},F_j\bigr]=\va_j(q_j-q_j^{-1})
        [F_j,F_k]_{q_k}E_jK_i^{-1}K_j^{-1}-q_j^{-1}\va_jF_kK_i^{-1}(1-K_j^{-2}),\\
        \bigl[\bigl[\,[F_i,F_j]_{q_j},F_k\bigr]_{q_k},G_j\bigr]=
            -\va_j(q_j-q_j^{-1})[F_j,F_k]_{q_k}E_jK_i^{-1}K_j^{-1}-q_j^{-1}\va_jF_kK_i^{-1}K_j^{-2},\\
            \bigl[\bigl[\,[F_i,G_j]_{q_j},G_k\bigr]_{q_k},F_j\bigr]
                =-q_j^{-1}\va_j\va_k [E_l,E_k]_{q_l^{-1}}K_i^{-1}K_k^{-1},
            \\
            \bigl[\bigl[\,[F_i,G_j]_{q_j},F_k\bigr]_{q_k},G_j\bigr]=0, \qquad
                 \bigl[\bigl[\,[F_i,F_j]_{q_j},G_k\bigr]_{q_k},G_j\bigr]=0,\qquad \bigl[\bigl[\,[F_i,G_j]_{q_j},G_k\bigr]_{q_k},G_j\bigr]=0.
    \end{gather*}
\end{lem}

\begin{proof}
By the quantum Serre relation \cref{Serre}, we have
\(
    \bigl[\bigl[\,[F_i,F_j]_{q_j},F_k\bigr]_{q_k},F_j\bigr]=0.
\)
Set
\(
    C:=[E_i,E_j]_{q_i^{-1}}.
\)
Since \(q_k=q_j^{-1}\), a direct calculation similar to \cref{lem:case2} gives
\begin{equation}
\label{case3shrimp1}
\begin{aligned}
    [F_i,G_j]_{q_j}
    &=
    \va_jE_jK_i^{-1}K_j^{-1},\\
    \bigl[\,[F_i,G_j]_{q_j},F_k\bigr]_{q_k}
    &=
    -\va_j(q_j-q_j^{-1})F_kE_jK_i^{-1}K_j^{-1},\\
    \bigl[\,[F_i,G_j]_{q_j},G_k\bigr]_{q_k}
    &=
    q_j^{-1}\va_j\va_k
    \bigl[E_j,[E_l,E_k]_{q_l^{-1}}\bigr]_{q_k}
    K_i^{-1}K_j^{-1}K_k^{-1}.
\end{aligned}
\end{equation}
Moreover, since \([E_l,E_k]_{q_l^{-1}}\) super-commutes with both
\(F_i\) and \(F_j\), and the \(K_k^{-1}\)-factor contributes exactly the
scalar killed by the \(q_k\)-commutator, we have
\begin{equation}
\label{case3shrimp2}
    \bigl[\,[F_i,F_j]_{q_j},G_k\bigr]_{q_k}=0.
\end{equation}
Therefore
\[
    \bigl[\bigl[\,[F_i,F_j]_{q_j},G_k\bigr]_{q_k},F_j\bigr]=0,
    \qquad
    \bigl[\bigl[\,[F_i,F_j]_{q_j},G_k\bigr]_{q_k},G_j\bigr]=0.
\]

We now compute the remaining terms. From \eqref{case3shrimp1},
\[
\begin{aligned}
\bigl[\bigl[\,[F_i,G_j]_{q_j},F_k\bigr]_{q_k},F_j\bigr]
&=
-\va_j(q_j-q_j^{-1})
\bigl[
    F_kE_jK_i^{-1}K_j^{-1},
    F_j
\bigr] \\
&=
-\va_j(q_j-q_j^{-1})
\left(
    q_j^{-1}F_kE_jF_j-F_jF_kE_j
\right)K_i^{-1}K_j^{-1}.
\end{aligned}
\]
Using
\(
    E_jF_j=-F_jE_j+\frac{K_j-K_j^{-1}}{q_j-q_j^{-1}},
\)
this becomes
\[
\begin{aligned}
\bigl[\bigl[\,[F_i,G_j]_{q_j},F_k\bigr]_{q_k},F_j\bigr]
=
\va_j(q_j-q_j^{-1})
[F_j,F_k]_{q_k}E_jK_i^{-1}K_j^{-1} 
-q_j^{-1}\va_jF_kK_i^{-1}(1-K_j^{-2}).
\end{aligned}
\]

Next, we have 
\[
\begin{aligned}
\bigl[\bigl[\,[F_i,G_j]_{q_j},F_k\bigr]_{q_k},G_j\bigr]
&=
-\va_j^2(q_j-q_j^{-1})
\bigl[
    F_kE_jK_i^{-1}K_j^{-1},
    CK_j^{-1}
\bigr].
\end{aligned}
\]
Moving the \(K\)-factors gives
\[
\bigl[
    F_kE_jK_i^{-1}K_j^{-1},
    CK_j^{-1}
\bigr]
=
F_k\bigl(E_jC+q_jCE_j\bigr)K_i^{-1}K_j^{-2}.
\]
Since
\(
    E_jC+q_jCE_j=0,
\)
we conclude that
\[
    \bigl[\bigl[\,[F_i,G_j]_{q_j},F_k\bigr]_{q_k},G_j\bigr]=0.
\]

It remains to compute the term $\bigl[\bigl[\,[F_i,F_j]_{q_j},F_k\bigr]_{q_k},G_j\bigr]$. The factor \(K_j^{-1}\) commutes with
\(\bigl[\,[F_i,F_j]_{q_j},F_k\bigr]_{q_k}\), since
\(
    a_{ji}+a_{jj}+a_{jk}=0.
\)
Moreover, as in the proof of \cref{lem:case2}, we have
\[
\bigl[
    [F_i,F_j]_{q_j},
    C
\bigr]
=-\bigl(q_jF_jE_j+q_j^{-1}E_jF_j\bigr)K_i^{-1}+q_j^{-1}F_iE_iK_j-E_iK_jF_i.
\]
Therefore
\[
\begin{aligned}
\bigl[
    \bigl[\,[F_i,F_j]_{q_j},F_k\bigr]_{q_k},
    C
\bigr] 
 =&
\bigl(q_jF_jE_j+q_j^{-1}E_jF_j\bigr)K_i^{-1}F_k
-
q_kF_k\bigl(q_jF_jE_j+q_j^{-1}E_jF_j\bigr)K_i^{-1} \\
 =&
-
\Bigl(
    q_j[F_j,F_k]_{q_k}E_j
    -
    q_j^{-1}E_j[F_j,F_k]_{q_k}
\Bigr)K_i^{-1}.
\end{aligned}
\]
Using
\(
    -E_j[F_j,F_k]_{q_k}+[F_j,F_k]_{q_k}E_j=F_kK_j^{-1},
\)
we obtain
\[
\bigl[
    \bigl[\,[F_i,F_j]_{q_j},F_k\bigr]_{q_k},
    C
\bigr]  
 =
-(q_j-q_j^{-1})[F_j,F_k]_{q_k}E_jK_i^{-1}
-q_j^{-1}F_kK_i^{-1}K_j^{-1}.
\]
Hence
\[
\bigl[\bigl[\,[F_i,F_j]_{q_j},F_k\bigr]_{q_k},G_j\bigr]
=
-\va_j(q_j-q_j^{-1})
[F_j,F_k]_{q_k}E_jK_i^{-1}K_j^{-1} 
-q_j^{-1}\va_jF_kK_i^{-1}K_j^{-2}.
\]

From \eqref{case3shrimp1}, the factor
\(K_i^{-1}K_j^{-1}K_k^{-1}\) commutes with \(F_j\). Since
\([E_l,E_k]_{q_l^{-1}}\) super-commutes with \(F_j\), we get
\[
\begin{aligned}
\bigl[\bigl[\,[F_i,G_j]_{q_j},G_k\bigr]_{q_k},F_j\bigr]
&=
q_j^{-1}\va_j\va_k
\bigl[
    \bigl[E_j,[E_l,E_k]_{q_l^{-1}}\bigr]_{q_k},
    F_j
\bigr]
K_i^{-1}K_j^{-1}K_k^{-1} \\
&=
-q_j^{-1}\va_j\va_k
[E_l,E_k]_{q_l^{-1}}K_jK_i^{-1}K_j^{-1}K_k^{-1} \\
&=
-q_j^{-1}\va_j\va_k
[E_l,E_k]_{q_l^{-1}}K_i^{-1}K_k^{-1}.
\end{aligned}
\]
Here we use
\(
    \bigl[
        \bigl[E_j,[E_l,E_k]_{q_l^{-1}}\bigr]_{q_k},
        F_j
    \bigr]
    =
    -[E_l,E_k]_{q_l^{-1}}K_j.
\)
Finally, the equality
\[
    \bigl[\bigl[\,[F_i,G_j]_{q_j},G_k\bigr]_{q_k},G_j\bigr]=0
\]
follows from \cite[Lem.~4.5(3)]{SW25}.
This proves the lemma.
\end{proof}
Summing the eight terms in \cref{lem:case3}, we obtain
\[
\begin{aligned}
\bigl[\bigl[\,[B_i,B_j]_{q_j},B_k\bigr]_{q_k},B_j\bigr]
&=
\va_j(q_j-q_j^{-1})
[F_j,F_k]_{q_k}E_jK_i^{-1}K_j^{-1}
-q_j^{-1}\va_jF_kK_i^{-1}(1-K_j^{-2}) \\
&\quad
-\va_j(q_j-q_j^{-1})
[F_j,F_k]_{q_k}E_jK_i^{-1}K_j^{-1}
-q_j^{-1}\va_jF_kK_i^{-1}K_j^{-2} \\
&\quad
-q_j^{-1}\va_j\va_k [E_l,E_k]_{q_l^{-1}}K_i^{-1}K_k^{-1} \\
&=
-q_j^{-1}\va_jF_kK_i^{-1}
-q_j^{-1}\va_j\va_k [E_l,E_k]_{q_l^{-1}}K_i^{-1}K_k^{-1} \\
&=
-q_j^{-1}\va_j
\left(F_k+\va_k [E_l,E_k]_{q_l^{-1}}K_k^{-1}\right)K_i^{-1} \\
&=
-q_j^{-1}\va_j B_kK_i^{-1}.
\end{aligned}
\]
This proves \cref{case3}. The calculation for the local Satake diagram of the form
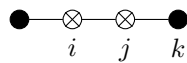
\[
\begin{tikzpicture}[anchorbase,scale=.7]
            \node (-2) [circ,fill=black,
                label={below:{}}] at (-1.5,0) {};
            \node (-1) [circcross,
                label={below:{$i$}}] at (-.5,0) {};
            \node (1) [circcross,
                label={below:{$j$}}] at (.5,0) {};
            \node (2) [circ,fill=black,
                label={below:{$k$}}] at (1.5,0) {};
            \path
                (-2) edge (-1)
                (-1) edge (1)
                (1) edge (2);
        \end{tikzpicture}
\]
is similar and omitted.\\

\subsection*{Competing interests}~ 
On behalf of all authors, the corresponding author states that there is no conflict of interest.


\bibliographystyle{alphaurl}
\bibliography{sAIII}

@article {AMS25,
    AUTHOR = {Algethami, D. and Mudrov, A. and Stukopin, V.},
     TITLE = {Quantum super-spherical pairs},
   JOURNAL = {J. Algebra},
  FJOURNAL = {Journal of Algebra},
    VOLUME = {674},
      YEAR = {2025},
    NUMBER = {2},
     PAGES = {276--313},
      ISSN = {0021-8693},
   MRCLASS = {17B37},
       DOI = {10.1016/j.jalgebra.2025.03.010},
       URL = {https://doi.org/10.1016/j.jalgebra.2025.03.010},
}

@article {L25,
    AUTHOR = {Lu, K.},
     TITLE = {Twisted super {Y}angians of type {AIII} and their
              representations},
   JOURNAL = {J. Algebra},
  FJOURNAL = {Journal of Algebra},
    VOLUME = {678},
      YEAR = {2025},
     PAGES = {74--132},
      ISSN = {0021-8693,1090-266X},
   MRCLASS = {17B37 (17B10)},
  MRNUMBER = {4897598},
MRREVIEWER = {Run-Qiang\ Jian},
       DOI = {10.1016/j.jalgebra.2025.04.008},
       URL = {https://doi-org.proxy.bib.uottawa.ca/10.1016/j.jalgebra.2025.04.008},
}

@article {L27,
    AUTHOR = {Lu, K.},
     TITLE = {Twisted super {Y}angians of quasi-split type {A}},
   JOURNAL = {J. Algebra},
  FJOURNAL = {Journal of Algebra},
    VOLUME = {711},
      YEAR = {2027},
     PAGES = {474--529},
       DOI = {10.1016/j.jalgebra.2026.06.035},
}

@article {CL23,
	AUTHOR = {Chen, J. and Luo, L.},
	TITLE = {Multiplication formulas and isomorphism theorem of
	{$\imath$}Schur superalgebras},
	JOURNAL = {J. Pure Appl. Algebra},
	FJOURNAL = {Journal of Pure and Applied Algebra},
	VOLUME = {227},
	YEAR = {2023},
	NUMBER = {3},
	PAGES = {Paper No. 107229, 36},
	ISSN = {0022-4049,1873-1376},
	MRCLASS = {20G43},
	MRNUMBER = {4478371},
	MRREVIEWER = {Mingqiang\ Liu},
	DOI = {10.1016/j.jpaa.2022.107229},
	URL = {https://doi-org.proxy.bib.uottawa.ca/10.1016/j.jpaa.2022.107229},
}

@article{SZ26,
    author = {Shen, Y. and Zhang, W.},
    title = {Relative braid group symmetries on quantum supersymmetric pairs of type s{AIII}},
    year = {2026},
ARCHIVEPREFIX = {arXiv},
    EPRINT = {2605.01696},
}

@article {Kol14,
    AUTHOR = {Kolb, S.},
     TITLE = {Quantum symmetric {K}ac-{M}oody pairs},
   JOURNAL = {Adv. Math.},
  FJOURNAL = {Advances in Mathematics},
    VOLUME = {267},
      YEAR = {2014},
     PAGES = {395--469},
      ISSN = {0001-8708,1090-2082},
   MRCLASS = {17B37 (17B67)},
  MRNUMBER = {3269184},
MRREVIEWER = {Darren\ Funk-Neubauer},
       DOI = {10.1016/j.aim.2014.08.010},
       URL = {https://doi.org/10.1016/j.aim.2014.08.010},
ARCHIVEPREFIX = {arXiv},
    EPRINT = {1207.6036},
}

@article {Let99,
    AUTHOR = {Letzter, G.},
     TITLE = {Symmetric pairs for quantized enveloping algebras},
   JOURNAL = {J. Algebra},
  FJOURNAL = {Journal of Algebra},
    VOLUME = {220},
      YEAR = {1999},
    NUMBER = {2},
     PAGES = {729--767},
      ISSN = {0021-8693},
   MRCLASS = {17B37},
       DOI = {10.1006/jabr.1999.8015},
       URL = {https://doi.org/10.1006/jabr.1999.8015},
}

@incollection {Let02,
    AUTHOR = {Letzter, G.},
     TITLE = {Coideal subalgebras and quantum symmetric pairs},
 BOOKTITLE = {New directions in {H}opf algebras},
    SERIES = {Math. Sci. Res. Inst. Publ.},
    VOLUME = {43},
     PAGES = {117--165},
 PUBLISHER = {Cambridge Univ. Press, Cambridge},
      YEAR = {2002},
   MRCLASS = {17B37},
ARCHIVEPREFIX = {arXiv},
    EPRINT = {math/0103228},
}

@book {Lus10,
    AUTHOR = {Lusztig, G.},
     TITLE = {Introduction to {Q}uantum {G}roups},
    SERIES = {Modern Birkh\"{a}user Classics},
      NOTE = {Reprint of the 1994 edition},
 PUBLISHER = {Birkh\"{a}user/Springer, New York},
      YEAR = {2010},
     PAGES = {xiv+346},
      ISBN = {978-0-8176-4716-2},
   MRCLASS = {17B37 (16T05 17-02 17B35)},
  MRNUMBER = {2759715},
       DOI = {10.1007/978-0-8176-4717-9},
       URL = {https://doi.org/10.1007/978-0-8176-4717-9},
}

@article {SSS26,
    AUTHOR = {Salmasian, H. and Savage, A. and Shen, Y.},
     TITLE = {The disoriented skein and iquantum {B}rauer categories},
   JOURNAL = {Forum Math. Sigma},
  FJOURNAL = {Forum of Mathematics. Sigma},
    VOLUME = {14},
      YEAR = {2026},
     PAGES = {1--41},
       DOI = {https://doi.org/10.1017/fms.2025.10148},
}

@article {She25,
    AUTHOR = {Shen, Y.},
     TITLE = {Quantum supersymmetric pairs and {$\imath$}{S}chur duality of
              type {AIII}},
   JOURNAL = {J. Algebra},
  FJOURNAL = {Journal of Algebra},
    VOLUME = {661},
      YEAR = {2025},
     PAGES = {853--903},
      ISSN = {0021-8693,1090-266X},
   MRCLASS = {17B37 (20C08)},
  MRNUMBER = {4793407},
       DOI = {10.1016/j.jalgebra.2024.07.035},
       URL = {https://doi.org/10.1016/j.jalgebra.2024.07.035},
ARCHIVEPREFIX = {arXiv},
    EPRINT = {2210.01233},
}

@article {SW25,
    AUTHOR = {Shen, Y. and Wang, W.},
     TITLE = {Quantum supersymmetric pairs of basic types},
   JOURNAL = {Comm. Math. Phys.},
  FJOURNAL = {Communications in Mathematical Physics},
    VOLUME = {406},
      YEAR = {2025},
    NUMBER = {8},
     PAGES = {Paper No. 187},
      ISSN = {0010-3616,1432-0916},
   MRCLASS = {99-06},
  MRNUMBER = {4927816},
       DOI = {10.1007/s00220-025-05339-w},
       URL = {https://doi.org/10.1007/s00220-025-05339-w},
ARCHIVEPREFIX = {arXiv},
    EPRINT = {2408.02874},
}

@incollection {Wan23, 
    AUTHOR = {Wang, W.},
     TITLE = {Quantum symmetric pairs},
 BOOKTITLE = {I{CM}---{I}nternational {C}ongress of {M}athematicians. {V}ol.
              4. {S}ections 5--8},
     PAGES = {3080--3102},
 PUBLISHER = {EMS Press, Berlin},
      YEAR = {2023},
      ISBN = {978-3-98547-062-4; 978-3-98547-562-9; 978-3-98547-058-7},
   MRCLASS = {17B37 (17B10 20G42)},
  MRNUMBER = {4680353},
MRREVIEWER = {J\"org\ Feldvoss},
ARCHIVEPREFIX = {arXiv},
    EPRINT = {2112.10911},
}

@article {Wat21,
    AUTHOR = {Watanabe, H.},
     TITLE = {Classical weight modules over {$\imath$}quantum groups},
   JOURNAL = {J. Algebra},
  FJOURNAL = {Journal of Algebra},
    VOLUME = {578},
      YEAR = {2021},
     PAGES = {241--302},
      ISSN = {0021-8693,1090-266X},
   MRCLASS = {17B37 (17B10)},
  MRNUMBER = {4234802},
MRREVIEWER = {Iwan\ Praton},
       DOI = {10.1016/j.jalgebra.2021.02.023},
       URL = {https://doi.org/10.1016/j.jalgebra.2021.02.023},
ARCHIVEPREFIX = {arXiv},
    EPRINT = {1912.11157},
}

@article {Yam94,
    AUTHOR = {Yamane, H.},
     TITLE = {Quantized enveloping algebras associated with simple {L}ie
              superalgebras and their universal {$R$}-matrices},
   JOURNAL = {Publ. Res. Inst. Math. Sci.},
  FJOURNAL = {Kyoto University. Research Institute for Mathematical
              Sciences. Publications},
    VOLUME = {30},
      YEAR = {1994},
    NUMBER = {1},
     PAGES = {15--87},
      ISSN = {0034-5318,1663-4926},
   MRCLASS = {17B37 (17B35)},
  MRNUMBER = {1266383},
MRREVIEWER = {Toshiyuki\ Tanisaki},
       DOI = {10.2977/prims/1195166275},
       URL = {https://doi.org/10.2977/prims/1195166275},
}

@article {Yam99,
    AUTHOR = {Yamane, H.},
     TITLE = {On defining relations of affine {L}ie superalgebras and affine
              quantized universal enveloping superalgebras},
   JOURNAL = {Publ. Res. Inst. Math. Sci.},
  FJOURNAL = {Kyoto University. Research Institute for Mathematical
              Sciences. Publications},
    VOLUME = {35},
      YEAR = {1999},
    NUMBER = {3},
     PAGES = {321--390},
      ISSN = {0034-5318,1663-4926},
   MRCLASS = {17B37 (17B67)},
  MRNUMBER = {1710748},
MRREVIEWER = {Naihuan\ Jing},
       DOI = {10.2977/prims/1195143607},
       URL = {https://doi.org/10.2977/prims/1195143607},
ARCHIVEPREFIX = {arXiv},
    EPRINT = {q-alg/9603015},
}

\end{document}